\documentclass{amsart}
\title[Non-commutative thickening]{Non-commutative thickening of moduli spaces of stable sheaves}

\date{}
\author{Yukinobu Toda}

\usepackage{amscd}
\usepackage{amsmath}
\usepackage{amssymb}
\usepackage{amsthm}
\usepackage{float}
\usepackage[dvips]{graphicx}

\usepackage[all,ps,dvips]{xy}

\usepackage{array}
\usepackage{amscd}
\usepackage[all]{xy}

\DeclareFontFamily{U}{rsfs}{%
\skewchar\font127}
\DeclareFontShape{U}{rsfs}{m}{n}{%
<-6>rsfs5<6-8.5>rsfs7<8.5->rsfs10}{}
\DeclareSymbolFont{rsfs}{U}{rsfs}{m}{n}
\DeclareSymbolFontAlphabet
{\mathrsfs}{rsfs}
\DeclareRobustCommand*\rsfs{%
\@fontswitch\relax\mathrsfs}

\theoremstyle{plain}
\newtheorem{thm}{Theorem}[section]
\newtheorem{prop}[thm]{Proposition}
\newtheorem{lem}[thm]{Lemma}

\newtheorem{defi}[thm]{Definition}
\newtheorem{rmk}[thm]{Remark}
\newtheorem{cor}[thm]{Corollary}

\newtheorem{prop-defi}[thm]{Proposition-Definition}
\newtheorem{thm-defi}[thm]{Theorem-Definition}
\newtheorem{lem-defi}[thm]{Lemma-Definition}
\newtheorem{problem}[thm]{Problem}

\newtheorem{assum}[thm]{Assumption}

\newdimen\argwidth
\def\db[#1\db]{
 \setbox0=\hbox{$#1$}\argwidth=\wd0
 \setbox0=\hbox{$\left[\box0\right]$}
  \advance\argwidth by -\wd0
 \left[\kern.3\argwidth\box0 \kern.3\argwidth\right]}

\newcommand{\aA}{\mathcal{A}}

\newcommand{\cC}{\mathcal{C}}

\newcommand{\eE}{\mathcal{E}}
\newcommand{\fF}{\mathcal{F}}
\newcommand{\gG}{\mathcal{G}}
\newcommand{\hH}{\mathcal{H}}
\newcommand{\iI}{\mathcal{I}}
\newcommand{\jJ}{\mathcal{J}}

\newcommand{\lL}{\mathcal{L}}
\newcommand{\mM}{\mathcal{M}}
\newcommand{\nN}{\mathcal{N}}
\newcommand{\oO}{\mathcal{O}}
\newcommand{\pP}{\mathcal{P}}

\newcommand{\sS}{\mathcal{S}}
\newcommand{\tT}{\mathcal{T}}

\newcommand{\vV}{\mathcal{V}}
\newcommand{\wW}{\mathcal{W}}

\newcommand{\Hom}{\mathop{\rm Hom}\nolimits}

\newcommand{\dotimes}{\stackrel{\textbf{L}}{\otimes}}
\newcommand{\dR}{\mathbf{R}}
\newcommand{\dL}{\mathbf{L}}

\newcommand{\Hilb}{\mathop{\rm Hilb}\nolimits}

\newcommand{\Pic}{\mathop{\rm Pic}\nolimits}

\newcommand{\id}{\textrm{id}}

\newcommand{\Ext}{\mathop{\rm Ext}\nolimits}
\newcommand{\Spec}{\mathop{\rm Spec}\nolimits}
\newcommand{\Spf}{\mathop{\rm Spf}\nolimits}

\newcommand{\Coh}{\mathop{\rm Coh}\nolimits}

\newcommand{\cneq}{\mathrel{\raise.095ex\hbox{:}\mkern-4.2mu=}}
\newcommand{\eqcn}{\mathrel{=\mkern-4.5mu\raise.095ex\hbox{:}}}

\newcommand{\gr}{\mathop{\rm gr}\nolimits}

\newcommand{\Aut}{\mathop{\rm Aut}\nolimits}

\newcommand{\Sym}{\mathop{\rm Sym}\nolimits}
\newcommand{\modu}{\mathop{\rm mod}\nolimits}

\newcommand{\Imm}{\mathop{\rm Im}\nolimits}

\newcommand{\Ker}{\mathop{\rm Ker}\nolimits}

\newcommand{\GL}{\mathop{\rm GL}\nolimits}

\newcommand{\n}{\mathrm{nc}}

\makeatletter
 
  \@addtoreset{equation}{subsection}
\makeatother

\begin{document}
\maketitle

\begin{abstract}
We 
show that the moduli spaces of 
stable sheaves on 
projective schemes admit 
certain non-commutative structures,  
which we call quasi NC structures, 
generalizing Kapranov's NC structures. 
The completion of our quasi NC structure 
at a closed point of the moduli space 
gives a pro-representable hull 
of the non-commutative deformation functor 
of the corresponding sheaf 
developed by Laudal, Eriksen, Segal and 
Efimov-Lunts-Orlov. 
We also show that the framed stable moduli spaces of 
sheaves have canonical NC structures. 
\end{abstract}

\section{Introduction}
\subsection{Background}
Let $X$ be a projective scheme 
over $\mathbb{C}$ and $\alpha$
a Hilbert polynomial of some coherent sheaf on $X$. 
It is well-known that (cf.~\cite{Hu}) the isomorphism 
classes of stable sheaves
\begin{align*}
M_{\alpha} \cneq \{\mbox{Stable sheaves on }
X \mbox{ with Hilbert polynomial }
 \alpha\}/(\mbox{isom})
\end{align*}
has a structure of a quasi projective scheme, 
called the \textit{moduli space of stable sheaves}.  
Under some primitivity condition of $\alpha$, 
which we always assume 
in this paper, 
the scheme $M_{\alpha}$ is projective and 
represents the functor of flat families of stable sheaves
on $X$
with Hilbert polynomial $\alpha$.

The moduli space $M_{\alpha}$
plays important roles in several 
places of algebraic geometry, e.g. 
the constructions of holomorphic symplectic manifolds~\cite{Mu2}, 
Fourier-Mukai transforms~\cite{Mu1}
and Donaldson-Thomas invariants~\cite{Thom}.
Recently the moduli scheme $M_{\alpha}$ turned out to 
be the truncation of 
a smooth derived moduli 
scheme~\cite{Toen2}, \cite{BFHR}, 
which has a shifted 
symplectic structure 
if $X$ is a Calabi-Yau manifold~\cite{PTVV}.
The shifted symplectic structure in~\cite{PTVV}
was used in~\cite{BBBJ}
to construct algebraic 
Chern Simons functions describing
$M_{\alpha}$ locally as its critical locus, 
which is crucial 
in the wall-crossing of DT invariants~\cite{JS}, \cite{K-S}. 
It is now an important subject to 
find such hidden structures on the moduli spaces
of stable sheaves, and give applications to 
an enumerative geometry. 
The goal of this paper is 
to construct such a hidden structure 
on $M_{\alpha}$ in a different direction, 
that is a certain \textit{non-commutative 
structure}, whose existence was known only  
at the formal level. 
\subsection{Non-commutative deformations of sheaves}
Let us recall formal deformation theory of sheaves 
in terms of dg-algebras and $A_{\infty}$-algebras.
For example, we refer to the Segal's paper~\cite{ESe}
on some details of this subject. 
 
Let $F$ be a stable sheaf on $X$ giving a closed 
point of $M_{\alpha}$, 
and $\widehat{\oO}_{M_{\alpha}, [F]}$
the completion of $\oO_{M_{\alpha}}$ 
at the closed point $[F] \in M_{\alpha}$. 
The algebra $\widehat{\oO}_{M_{\alpha}, [F]}$
pro-represents
the commutative deformation functor
\begin{align*}
\mathrm{Def}_{F} \colon 
\aA rt^{\rm{loc}} \to \sS et
\end{align*}
sending $R$ to the 
isomorphism classes of flat deformations of $F$ to 
$\fF \in \Coh(X \times \Spec R)$. 
Here $\aA rt^{\rm{loc}}$ is the category of
local Artinian $\mathbb{C}$-algebras. 
It is well-known that 
the functor $\mathrm{Def}_F$
is governed by the 
dg-algebra $\dR \Hom(F, F)$, that 
is 
\begin{align*}
\mathrm{Def}_F(R)=\mathrm{MC}(\dR \Hom(F, F) \otimes \bf{m})/(\mbox{gauge equivalence})
\end{align*}
where ${\bf m} \subset R$ is 
the maximal ideal 
and $\mathrm{MC}(\mathfrak{g}^{\bullet})$ 
for the dg-Lie algebra $\mathfrak{g}^{\bullet}$
is the 
solution of the Maurer-Cartan equation
\begin{align*}
\mathrm{MC}(\mathfrak{g}^{\bullet})
=\left\{ x \in \mathfrak{g}^1 : dx+\frac{1}{2}[x, x]=0 \right\}.
\end{align*}
On the other hand, by the minimal model theorem of $A_{\infty}$-algebras, 
we have a quasi-isomorphism
\begin{align*}
\dR \Hom(F, F)  \sim (\Ext^{\ast}(F, F), \{m_n\}_{n\ge 2})
\end{align*}
where the RHS is 
a minimal $A_{\infty}$-algebra. 
The formal solution of the Mauer-Cartan equation 
of the RHS 
is then $\Spec R_{F}$, where $R_F$ is the commutative algebra
\begin{align}\label{intro:RF}
R_F
=\frac{\widehat{\Sym^{\bullet}}(\Ext^1(F, F)^{\vee})}{\left(\sum_{n\ge 2} 
\overline{m}_n^{\vee}
 \right)}.
\end{align}
Here $\overline{m}_n^{\vee}$ is the dual of the $A_{\infty}$-product
\begin{align}\label{intro:mn}
m_n^{\vee} \colon \Ext^2(F, F)^{\vee} \to \Ext^1(F, F)^{\vee \otimes n}
\end{align}
composed with the 
symmetrization map 
\begin{align}\label{intro:sym}
\Ext^1(F, F)^{\vee \otimes n}
\twoheadrightarrow \Sym^n(\Ext^1(F, F)^{\vee}).
\end{align}
By a general theory of deformation theory, 
$R_F$ also pro-represents $\mathrm{Def}_F$, 
hence we have the isomorphism
\begin{align*}
\widehat{\oO}_{M_{\alpha}, [F]} \cong R_F.
\end{align*}

However the 
dual of the $A_{\infty}$-products (\ref{intro:mn}) take values
in the tensor products
rather than symmetric products, 
so the algebra $R_F$ may lose some information
of the $A_{\infty}$-products
under the maps (\ref{intro:sym}). 
Instead of the commutative 
algebra $R_F$, 
the possibly non-commutative algebra
\begin{align}\label{intro:RFnc}
R_F^{\n} =\frac{\prod_{n\ge 0}(\Ext^1(F, F)^{\vee})^{\otimes n}}{(\sum_{n\ge 2} m_n^{\vee})}
\end{align}
is more natural and
keep the information of 
the $A_{\infty}$-products which the algebra $R_F$ may lose.  
In fact, the algebra (\ref{intro:RFnc}) appears
in the context of 
\textit{non-commutative deformation theory} of sheaves. 
The algebra (\ref{intro:RFnc}) is a
pro-representable hull of the non-commutative deformation functor
\begin{align}\label{intro:DefF}
\mathrm{Def}_{F}^{\n} \colon 
\nN^{\rm{loc}} \to \sS et
\end{align}
where $\nN^{\rm{loc}}$ is the category of 
finite dimensional (not necessary commutative) local 
$\mathbb{C}$-algebras.
The functor
$\mathrm{Def}_F^{\n}$ 
sends $\Lambda \in \nN^{\rm{loc}}$ to the isomorphism
classes of flat deformations of $F$ to 
$\fF \in \Coh(\oO_X \otimes_{\mathbb{C}} \Lambda)$. 
Such a non-commutative deformation theory was 
studied by Laudal~\cite{Lau} 
for modules over algebras, and 
later developed by Eriksen~\cite{Erik}, Segal~\cite{ESe}
and Efimov-Lunts-Orlov~\cite{ELO}, \cite{ELO2}, \cite{ELO3}
in geometric contexts. 

\subsection{Global non-commutative moduli spaces of stable sheaves}
Note that the algebra $R_F^{\n}$ reconstructs
$R_F$ by taking its abelization. 
Therefore at the formal level, 
we see that 
the structure sheaf $\oO_{M_{\alpha}}$ admits 
a possibly non-commutative 
enhancement $R_F^{\rm{nc}}$
in the sense that
\begin{align*}
(R_F^{\rm{nc}})^{ab} \cong \widehat{\oO}_{M_{\alpha}, [F]}. 
\end{align*}
The purpose of this paper is to 
give a globalization of the above isomorphism.  
Namely, we would like to construct
a sheaf of non-commutative algebras 
on $M_{\alpha}$
whose formal completion at $F$
gives the algebra (\ref{intro:RFnc}). 
We formulate this problem using the 
notion of Kapranov's 
\textit{NC schemes}~\cite{Kap17}.
Roughly speaking, an NC scheme is a ringed space 
$(Y, \oO_Y)$
whose structure sheaf $\oO_Y$ 
is possibly non-commutative, 
that is formal in the non-commutative direction.
Its abelization $(Y, \oO_Y^{ab})$ is 
a usual scheme, and $(Y, \oO_Y)$ is interpreted as a 
formal non-commutative thickening of $(Y, \oO_Y^{ab})$. 
We call an NC scheme $(Y, \oO_Y)$
 as an \textit{NC structure} on $(Y, \oO_Y^{ab})$. 
 The following problem 
is the main interest in this paper: 
\begin{problem}\label{intro:problem}
Is there a NC structure 
$(M_{\alpha}, \oO_{M_{\alpha}}^{\n})$ on $M_{\alpha}$ such
that $\widehat{\oO}_{M_{\alpha}, [F]}^{\n} \cong R_F^{\n}$
for any $[F] \in M_{\alpha}$ ?
\end{problem}
The above problem was
addressed by Kapranov~\cite{Kap17}
when any $[F] \in M_{\alpha}$ 
is a vector bundle without 
obstruction space\footnote{However
it was pointed out in~\cite[Remark~4.1.4]{PoTu}
that the proof of~\cite[Proposition~5.4.3]{Kap17}
has a gap. So we do not know whether 
Problem~\ref{intro:problem}
is true or not when any $[F] \in M_{\alpha}$
is a vector bundle without obstruction space.
Also see Remark~\ref{rmk:aut}.}, i.e. $\Ext^2(F, F)=0$, 
and by Polishchuk-Tu~\cite{PoTu}
when $M_{\alpha}$ is the moduli space of line bundles. 
The main result of this paper, 
stated in the following theorem,
 is to prove a weaker version of 
Problem~\ref{intro:problem}:
\begin{thm}\emph{(Theorem~\ref{main:thm})}\label{intro:main}
There exists an affine open covering $\{U_i\}_{i\in \mathbb{I}}$
of $M_{\alpha}$, NC structures $U_{i}^{\n}=(U_i, \oO_{U_i}^{\n})$
on each $U_i$, isomorphisms of NC schemes 
\begin{align}\label{intro:qnc}
\phi_{ij} \colon U_i^{\n}|_{U_{ij}} \stackrel{\cong}{\to} U_j^{\n}|_{U_{ij}},
\ \phi_{ij}^{ab}=\id
\end{align}
such that 
for any $[F] \in U_i$, 
we have $\widehat{\oO}_{U_i, [F]}^{\n} \cong R_F^{\n}$. 
\end{thm}
We call 
NC structures $\{U_i^{\n}\}_{i\in \mathbb{I}}$
satisfying the condition (\ref{intro:qnc}) as a
\textit{quasi NC structure} on $M_{\alpha}$. 
It gives a global NC structure on $M_{\alpha}$ if the isomorphisms
(\ref{intro:qnc}) satisfy the cocycle condition. 
The obstruction 
for this cocycle condition
lies in $H^2$ sheaf cohomology of $M_{\alpha}$, 
hence Theorem~\ref{intro:main}
implies Problem~\ref{intro:problem}
if $\dim M_{\alpha} \le 1$. 
The quasi NC structure (\ref{intro:qnc}) may be 
treated as a twisted sheaf, 
and the gluing problem 
of $\{U_i\}_{i\in \mathbb{I}}$
is 
something similar to the existence problem of the universal 
sheaf over $X \times M_{\alpha}$. 
In any case, the isomorphisms $\phi_{ij}$
satisfy the cocycle condition 
after taking the 
subquotients of the NC filtration, defined in Subsection~\ref{subsec:NCnil}.
Hence we obtain the commutative scheme over $M_{\alpha}$
(cf.~Remark~\ref{rmk:glue})
\begin{align}\label{intro:scheme}
\bigcup_{i \in \mathbb{I}}
\Spec \left(\mathrm{gr}_{F}^{\bullet}(\oO_{U_i}^{\n}) \right) \to M_{\alpha}
\end{align}
which is canonically attached to $M_{\alpha}$. 
The scheme (\ref{intro:scheme}) 
provides a new geometric structure of $M_{\alpha}$ which 
captures the non-commutative 
deformation theory of sheaves. 

As we will see in Remark~\ref{rmk:aut}, 
the gluing issue
in Theorem~\ref{intro:main}
is caused by 
the possible existence of automorphisms
of flat families of stable sheaves over a non-commutative base
which do not extend to automorphisms of 
further deformations. 
If we instead consider a moduli problem without any 
automorphism, then this issue is fixed.  
One of the classical ways to kill automorphisms
of sheaves
is to add data of \textit{framing}. 
For an integer $p$, we
consider the moduli space 
$M_{\alpha, p}^{\star}$
of 
pairs $(F, s)$, where 
$F$ is a semistable sheaf with Hilbert polynomial $\alpha$, 
$s$ is an element (called framing)
of $H^0(F(p))$, satisfying some stability condition. 
We have the following result: 
\begin{thm}\emph{(Theorem~\ref{thm:NCframe})}\label{intro:main2}
For $p\gg 0$, the 
framed moduli scheme $M_{\alpha, p}^{\star}$
has a canonical NC structure.  
\end{thm}
Similarly to Theorem~\ref{intro:main}, 
the NC structure 
in Theorem~\ref{intro:main2}
locally represents the 
non-commutative deformation functor of 
framed sheaves. 
As a corollary of Theorem~\ref{intro:main2}, 
we have the canonical NC structure 
on the Hilbert scheme $\Hilb^n(X)$
of $n$-points on $X$ (cf.~Subsection~\ref{subsec:hilb}). 

\subsection{Idea of the proof}
As described in~\cite{Hu}, 
the classical way to construct 
the moduli space $M_{\alpha}$ is to take 
the GIT quotient of the
Grothendieck Quot scheme. 
Rather recently, 
\'Alvarez-C\'onsul-King~\cite{ACK}
gave another construction of $M_{\alpha}$ using 
the moduli space of representations of 
a quiver. 
The argument of~\cite{ACK} 
was improved by
Behrend-Fontanine-Hwang-Rose~\cite{BFHR}
to construct 
the derived moduli scheme of stable sheaves. 
Our strategy 
is to construct a quasi NC structure on 
the moduli space of representations of a quiver, 
and pull it back to $M_{\alpha}$
by extending the arguments of~\cite{BFHR}
to non-commutative thickenings. 
 
The basic idea is 
as follows: we 
associate a stable sheaf $[F] \in M_{\alpha}$ 
with the vector space
\begin{align}\label{intro:Gpq}
\Gamma_{[p, q]}(F) 
\cneq \bigoplus_{i=p}^{q} \Gamma(X, F(i)). 
\end{align}
The vector space $\Gamma_{[p, q]}(F)$ is a representation of 
a certain quiver $Q_{[p, q]}$ with relation $I$,  
defined by the graded algebra structure 
of the homogeneous
coordinate ring of $X$
(cf.~Figure~\ref{fig:quiver}). 
For $q \gg p \gg 0$, 
the result of~\cite{BFHR} shows 
the map $F \mapsto \Gamma_{[p, q]}(F)$
gives an open immersion
\begin{align*}
\Gamma_{[p, q]} \colon M_{\alpha} \subset M_{Q, I}
\end{align*} 
where 
$M_{Q, I}$
is the 
moduli space of
representations of $(Q_{[p, q]}, I)$. 
We then construct 
a quasi NC structure on $M_{Q, I}$
in the following way.  
We first embed $M_{Q, I}$ into 
the smooth 
scheme
\begin{align*}
M_{Q, I} \subset M_{Q}
\end{align*}
where $M_{Q}$
is the moduli space of 
representations of $Q$ without relation. 
Let $U \subset M_{Q}$ be an 
affine open subset. 
By Kapranov~\cite{Kap17}, the smooth affine scheme $U$
 has an NC smooth thickening
$(U, \oO_{U}^{\n})$, 
which is unique up to non-canonical 
isomorphisms. 
We construct the 
NC structure
 on $V=U \cap M_{Q, I}$
by taking the 
quotient algebra $\oO_{U}^{\n}/\jJ_{I, U}$
to be its NC structure sheaf, 
where $\jJ_{I, U} \subset \oO_{U}^{\n}$ is the two 
sided ideal 
determined by the relations in $I$. 
We prove that the above local NC structures
\begin{align*}
V^{\n}=(V, \oO_{U}^{\n}/\jJ_{I, U})
\end{align*} 
on $M_{Q, I}$
are pulled back to $M_{\alpha}$ to give 
a desired quasi NC structure. 
The key technical 
ingredient is 
to construct the derived 
left adjoint functor of (\ref{intro:Gpq})
for flat families of $(Q_{[p, q]}, I)$-representations 
over non-commutative bases, which enables us to compare 
non-commutative deformations of sheaves with 
those of representations of $(Q_{[p, q]}, I)$. 

\subsection{Possible applications}
Let $X \to Y$ be 
a flopping contraction 
from a smooth 3-fold $X$, 
whose exceptional locus is 
a smooth rational curve $C$. 
In the paper~\cite{WM}, Donovan-Wemyss 
used
the algebra $R_{\oO_C}^{\n}$
to construct a non-commutative twist functor 
of $D^b \Coh(X)$, which describes 
Bridgeland-Chen's flop-flop autoequivalence~\cite{Br1}, \cite{Ch}. 
In this situation, the underlying commutative 
moduli space of $\oO_C$ is topologically one point, 
and the formal non-commutative 
moduli space 
$\Spf R_{\oO_C}^{\n}$ also 
gives the global one.  
On the other hand, one may try to construct a similar 
autoequivalence associated to a divisorial contraction 
$X \to Y$ 
for a Calabi-Yau 3-fold $X$, 
contracting a divisor $E \subset X$ to a curve $Z \subset Y$.  
Since $\dim Z =1$, 
Theorem~\ref{intro:main}
should yield a NC structure $Z^{\n}$ on $Z$
induced by non-commutative deformation theory of 
fibers of $E \to Z$. 
Then similarly to the work~\cite{WM}, 
one may able to construct 
an autoequivalence of $D^b \Coh(X)$
generalizing EZ-spherical twist~\cite{Ho}, 
using the NC structure $Z^{\n}$ on $Z$.  
When $Y$ is a spectrum of a complete local 
$\mathbb{C}$-algebra, 
such an autoequivalence was 
constructed by Wemyss~\cite[Section~4.4]{Wemy}, and 
our construction of $Z^{\n}$ should 
give a globalization of his result. 

In the above flopping
contraction case, the algebra 
$R_{\oO_C}^{\n}$ is finite dimensional. 
In the paper~\cite{Todw},
the author described the dimension of $R_{\oO_C}^{\n}$
in terms of Katz's genus zero Gopakumar-Vafa invariants 
on $X$.  
This phenomena suggests
that there might be a 
DT type theory which captures non-commutative deformations of sheaves
on Calabi-Yau 3-folds, and 
has some relations to the usual DT theory. 
The result of Theorem~\ref{intro:main} 
may lead to a foundation of such a theory. 
Indeed when 
there are no higher obstruction spaces, 
the construction of quasi NC structures in
Theorem~\ref{intro:main}
yields interesting virtual fundamental sheaves 
in $K_0(M_{\alpha})$ which captures 
non-commutative deformations of sheaves. 
The details of their constructions
and the properties will pursued in 
the next paper~\cite{Todnc2}.   

\subsection{Plan of the paper}
In Section~\ref{sec:NC}, 
we review Kapranov's NC schemes, 
introduce quasi NC structures, 
discuss the related notions. 
In Section~\ref{sec:NC2}, 
we construct quasi NC structures 
on the moduli spaces of representations of quivers
with relations. 
In Section~\ref{sec:NC3}, 
we 
construct quasi NC structures on the 
moduli spaces of stable sheaves
using those on representations of quivers
with relations. 
In Section~\ref{sec:NC4}, 
we describe some examples. 
\subsection{Acknowledgement}
The author would like to thank 
Tomoyuki Abe, Will Donovan, Zheng Hua 
and Michael Wemyss
for the discussions related to this paper. 
This work is supported by World Premier 
International Research Center Initiative
(WPI initiative), MEXT, Japan, 
and Grant-in Aid
for Scientific Research grant (No.~26287002)
from the Ministry of Education, Culture,
Sports, Science and Technology, Japan.

\subsection{Notation and convention}
In this paper, all the varieties
or schemes
are defined over $\mathbb{C}$. 
The category $\cC om$ is 
the category of commutative $\mathbb{C}$-algebras. 
An algebra always means an associative, 
not necessary commutative,  
$\mathbb{C}$-algebra. 
For an algebra $\Lambda$, the 
category $\Lambda \modu$ is the category 
of finitely generated left $\Lambda$-modules. 
We denote by $\Lambda^{ab}$ the abelization of $\Lambda$, 
and for $N \in \Lambda \modu$
we denote $N^{ab} \cneq \Lambda^{ab} \otimes_{\Lambda} N$. 
Also for left $\Lambda$-module homomorphism
$\phi \colon N_1 \to N_2$, we 
set $\phi^{ab}=\Lambda^{ab} 
\otimes_{\Lambda} \phi \colon N_1^{ab} \to N_2^{ab}$.

\section{NC structures and quasi NC structures}\label{sec:NC}
In this section, 
we recall the definition of NC schemes 
introduced by Kapranov~\cite{Kap17}.
We also introduce quasi NC structures, NC hulls 
and prove a way to construct them via functors.  
\subsection{NC nilpotent algebras}\label{subsec:NCnil}
Let $\Lambda$ be an algebra. 
We regard $\Lambda$ as a Lie algebra 
by setting $[a, b]=ab-ba$. 
The subspace
 $\Lambda_{k}^{\rm{Lie}} \subset \Lambda$
is defined to be spanned by the elements of the form
\begin{align*}
[x_1, [x_2, \cdots, [x_{k-1}, x_k]\cdots ]]
\end{align*}
for $x_i \in \Lambda$, $1\le i\le k$. 
The \textit{NC filtration} of $\Lambda$ is the decreasing 
filtration
\begin{align*}
\Lambda=F^0\Lambda \supset F^1 \Lambda \supset \cdots
\supset F^{d} \Lambda \supset \cdots
\end{align*}
where $F^d \Lambda$ is the two-sided ideal of $\Lambda$
defined by
\begin{align*}
F^d \Lambda \cneq 
\sum_{m\ge 0}\sum_{i_1+\cdots+i_m=m+d}\Lambda \cdot 
\Lambda_{i_1}^{\rm{Lie}} \cdot \Lambda \cdot 
\cdots \cdot \Lambda_{i_m}^{\rm{Lie}} \cdot \Lambda. 
\end{align*}
Note that
$\Lambda/F^1 \Lambda$ is the abelization 
$\Lambda^{ab}$ of $\Lambda$. 
We set $\Lambda^{\le d} \cneq \Lambda/F^{d+1} \Lambda$, 
and denote $N^{\le d} \cneq \Lambda^{\le d} \otimes_{\Lambda}N$
for $N \in \Lambda \modu$. 
\begin{defi}
(i) 
An algebra $\Lambda$
is called NC nilpotent of degree $d$
(resp.~NC nilpotent)
if $F^{d+1}\Lambda=0$
(resp.~$F^n \Lambda=0$ for $n\gg 0$).  

(ii) The NC completion of an algebra $\Lambda$ is 
\begin{align*}
\Lambda_{[[ab]]} \cneq
\lim_{\longleftarrow}\Lambda^{\le d}. 
\end{align*}

(iii) An algebra $\Lambda$ 
is called NC complete if the natural map 
$\Lambda \to \Lambda_{[[ab]]}$ is an isomorphism. 
\end{defi}

For an algebra $\Lambda$, its subquotient
\begin{align*}
\gr_{F}^{\bullet}(\Lambda) \cneq 
\bigoplus_{d\ge 0} F^d \Lambda/F^{d+1} \Lambda
\end{align*}
is a commutative algebra. 
\begin{lem}\label{lem:id}
For an algebra $\Lambda$ and 
an algebra homomorphism $\phi \colon \Lambda \to \Lambda$, 
suppose that the induced homomorphism 
$\phi^{ab} \colon \Lambda^{ab} \to \Lambda^{ab}$ is identity. 
Then $\gr_F^{\bullet}(\phi)=\id$, and if 
$\Lambda$ is NC complete, 
$\phi$ is an isomorphism. 
\end{lem}
\begin{proof}
The assumption $\phi^{ab}=\id$ implies that 
$\phi(x)-x \in F^1 \Lambda$
for any $x \in \Lambda$. Therefore 
$\phi(x)-x \in F^{d+1} \Lambda$ if $x \in F^d \Lambda$, 
hence $\gr_F^{\bullet}(\phi)=\id$. 
The fact $\gr_F^{\bullet}(\phi)=\id$ together with 
the induction on $d$ show that 
$\phi^{\le d} \colon \Lambda^{\le d} \to \Lambda^{\le d}$
is an isomorphism
for any $d>0$. 
Therefore $\phi$ is an isomorphism if $\Lambda$ is NC complete. 
\end{proof}

\subsection{NC schemes}
Let $\Lambda$ be an NC complete algebra. 
The \textit{affine NC scheme} 
\begin{align}\label{aNC}
\Spf \Lambda
=(\Spec \Lambda^{ab}, \oO_Y)
\end{align}
is a ringed space
defined in the following way. 
For any multiplicative set 
$S \subset \Lambda^{ab}$, 
its pull-back
by the natural surjection $\Lambda^{\le d} \twoheadrightarrow
 \Lambda^{ab}$
determines the multiplicative set 
in $\Lambda^{\le d}$, which 
satisfies the Ore localization condition (cf.~\cite[Proposition~2.1.5]{Kap17}). Therefore, similarly to the case of usual affine 
schemes, the NC nilpotent algebra
$\Lambda^{\le d}$ determines the sheaf of 
algebras $\oO_{Y}^{\le d}$
on $\Spec \Lambda^{ab}$. 
The sheaf of algebras $\oO_Y$ is defined by 
\begin{align*}
\oO_Y \cneq \lim_{\longleftarrow} \oO_{Y}^{\le d}. 
\end{align*}
\begin{defi}\emph{(\cite[Definition~2.2.5]{Kap17})}
A ringed space $Y$ is called 
an NC scheme if it is locally isomorphic to
an 
affine NC scheme of the form (\ref{aNC}). 
\end{defi}
The category of NC schemes is the full subcategory of 
ringed spaces consisting of NC schemes. 
For an NC scheme $Y$, the category $\Coh(Y)$ 
is defined to be the 
category of coherent left $\oO_Y$-modules on $Y$. 
Note that the structure sheaf $\oO_Y$ has a filtration 
by sheaves of two-sided ideals $F^n \oO_Y \subset \oO_Y$. 
The quotient $\oO_Y^{\le d} \cneq \oO_Y/F^{d+1}\oO_Y$ defines 
the NC subscheme $Y^{\le d} \subset Y$. 
In particular, $Y^{ab} \cneq Y^{\le 0}$ is 
a usual scheme. 
\begin{defi}
An NC structure on a commutative scheme $M$
is an NC scheme
\begin{align}\label{NCstr}
M^{\n}=(M, \oO_M^{\n})
\end{align}
such that $(\oO_M^{\n})^{ab}=\oO_M$. 
\end{defi}
Here in the RHS of (\ref{NCstr}), 
$M$ is just regarded as a topological space. 
In general, it is not easy 
to construct 
non-trivial NC structures on a given algebraic variety. 
Instead, we introduce the following weaker notion of 
NC structures:
\begin{defi}\label{def:QNC}
Let $M$ be a commutative scheme. 
A quasi NC structure on $M$
consists of an affine 
open cover $\{U_i\}_{i\in \mathbb{I}}$
of $M$, 
affine NC structures
$(U_{i}, \oO_{U_i}^{\rm{nc}})$
on $U_i$
for each $i\in \mathbb{I}$, 
and isomorphisms
\begin{align}\label{phiij}
\phi_{ij} \colon 
(U_{ij}, \oO_{U_j}^{\rm{nc}}|_{U_{ij}})
\stackrel{\cong}{\to}
(U_{ij}, \oO_{U_i}^{\rm{nc}}|_{U_{ij}}). 
\end{align}
satisfying $\phi_{ij}^{ab}=\id$. 
\end{defi}
\begin{rmk}
A quasi NC structure gives rise to the NC structure 
if and only if the isomorphisms $\phi_{ij}$ in (\ref{phiij})
satisfy the cocycle condition, i.e. 
\begin{align*}
\phi_{ij} \circ \phi_{jk} \circ \phi_{ki}=\id.
\end{align*}
\end{rmk}
\begin{rmk}\label{rmk:glue}
Although (\ref{phiij})
may not satisfy the cocycle condition, 
Lemma~\ref{lem:id} 
implies that $\gr_F^{\bullet}(\phi_{ij})$ 
always satisfies the 
cocycle condition.
Hence $\gr_F^{\bullet}(\oO_{U_i}^{\n})$ glue together
to give a sheaf of algebras on $M$.
In particular, we have the
commutative scheme over $M$
\begin{align*}
\bigcup_{i \in \mathbb{I}}
\Spec \left(\gr_F^{\bullet}(\oO_{U_i}^{\n})\right) \to M. 
\end{align*}
\end{rmk}
\subsection{Smooth NC schemes}
Let $\nN_d$ be the category of 
NC nilpotent algebras of degree $d$, and 
$\nN$ the category of NC nilpotent algebras. 
We have the following inclusions
\begin{align}\label{Com}
\cC om \cneq \nN_0 \subset \nN_1 \subset \cdots \subset \nN_d \subset \cdots 
\subset \nN. 
\end{align} 
Let $\Lambda' \twoheadrightarrow \Lambda$
be a surjection in $\nN$, and take the 
exact sequence
\begin{align}\label{central}
0 \to J \to \Lambda' \to \Lambda \to 0.
\end{align}
The sequence (\ref{central}) 
is called 
a \textit{central extension} if $J^2=0$ and 
$J$ lies in the center of $\Lambda'$. 
In particular, $J$ is a $\Lambda^{ab}$-module. 
\begin{defi}\label{defi:smooth}
For functors  
$h_1, h_2 \colon \nN \to \sS et$, 
a natural transform
$\phi \colon h_1 \to h_2$
is called 
formally smooth if 
for any
central extension (\ref{central})
in $\nN$, 
the map 
\begin{align*}
h_1(\Lambda') \to h_2(\Lambda') \times_{h_2(\Lambda)}
h_1(\Lambda)
\end{align*}
is surjective. 
\end{defi}
An NC scheme $Y$ defines a covariant functor
\begin{align*}
h_{Y} \colon 
\nN \to \sS et
\end{align*}
sending $\Lambda$ to 
$\Hom(\Spf \Lambda, Y)$. 
In the situation of Definition~\ref{defi:smooth},
a functor $h \colon \nN \to \sS et$ is 
called \textit{formally smooth} if the 
natural transform 
$h \to h_{\Spec \mathbb{C}}$ is 
formally smooth, i.e. 
for any central extension (\ref{central}), the 
map $h(\Lambda') \to h(\Lambda)$ is surjective.  
If the same condition holds 
only for
 central extensions (\ref{central}) 
in $\nN_d$,
we call
the functor $h|_{\nN_d}$
as \textit{formally} $d$-\textit{smooth}. 
\begin{defi}
(i) 
An NC scheme $Y$ is called smooth 
if $h_{Y}$ is formally smooth. 

(ii) An NC scheme $Y$
is called $d$-smooth 
if 
$F^{d+1}\oO_Y=0$ and 
$h_{Y}|_{\nN_d}$ is formally 
$d$-smooth. 
\end{defi}
It is easy to see that if an NC scheme $Y$
is smooth, then 
$Y^{\le d}$
is $d$-smooth. 
In particular, 
$Y_{ab}$ is a smooth scheme
in the usual sense. 
If $Y$ is NC smooth, we say that 
$Y$ is an \textit{NC smooth thickening} of $Y_{ab}$. 
By~\cite[Theorem~1.6.1]{Kap17}, 
any smooth affine scheme has 
an NC smooth thickening, which 
is unique up to non-canonical isomorphisms. 
In particular, any smooth 
algebraic variety has 
a smooth quasi NC structure 
in the sense of Definition~\ref{def:QNC}. 
 
\subsection{(Quasi) NC structures via functors}
We introduce the notion of NC hull
to construct quasi NC structures. 
\begin{defi}
Let $h \colon \nN \to \sS et$ be a functor. 
An NC scheme $Y$ together with
a natural transform $\phi \colon h_Y \to h$ is called 
 an NC hull of $h$ if 
$\phi$ is formally smooth and isomorphism on $\cC om$. 
If $Y$ is an affine NC scheme, we call it an affine NC hull. 
\end{defi}
We have the following lemma:
\begin{lem}\label{NCh}
An affine NC hull is unique up to non-canonical 
isomorphisms. 
\end{lem}
\begin{proof}
For $i=1, 2$, let 
$\phi_i \colon h_{Y_i} \to h$ are affine NC hulls, 
and write $Y_i=\Spf \Lambda_i$. 
Note that, by the formal smoothness of NC hulls, 
the natural map 
\begin{align}\label{limh}
\lim_{\longleftarrow}
h_{Y_i}(\Lambda^{\le d}) \to \lim_{\longleftarrow} h(\Lambda^{\le d})
\end{align}
is surjective for any NC complete algebra $\Lambda$. 
Since the LHS of (\ref{limh})
coincides with $\Hom(\Lambda_i, \Lambda)$, 
there exist algebra homomorphisms 
$u \colon \Lambda_1 \to \Lambda_2$, 
$v \colon \Lambda_2 \to \Lambda_1$ 
which commute with $\phi_1$, $\phi_2$. 
Then the compositions 
\begin{align}\label{uv}
v \circ u \colon \Lambda_1 \to \Lambda_1, \ 
u \circ v \colon \Lambda_2 \to \Lambda_2
\end{align}
satisfy $(v\circ u)^{ab}=\id$, 
$(u \circ v)^{ab}=\id$. 
Therefore Lemma~\ref{lem:id}
implies that (\ref{uv}) are algebra isomorphisms, 
hence $u$, $v$ are isomorphisms. 
\end{proof}
Let $M$ be a commutative scheme, 
and $h \colon \nN \to \sS et$ a functor. 
If $h|_{\cC om}=h_M$, 
we have the natural 
transform 
$h \to h_{M}$
applying $h$ to $\Lambda \to \Lambda^{ab}$ 
for $\Lambda \in \nN$. 
For a subscheme $U \subset M$, 
we define 
\begin{align*}
h|_{U} \cneq h \times_{h_M} h_U. 
\end{align*}
The following corollary obviously follows 
from Lemma~\ref{NCh}:
\begin{cor}\label{cor:sit}
In the above situation, 
suppose that 
$M$ has an affine open cover 
$\cup_{i\in \mathbb{I}}U_i$, 
affine NC structures $U_i^{\n}$
on each $U_i$, 
and NC hulls 
$h_{U_i^{\n}} \to h|_{U_i}$. 
Then $\{U_i^{\n}\}_{i\in \mathbb{I}}$ 
gives a quasi NC structure on $M$. 
\end{cor}
\begin{rmk}\label{rmk:NC}
In the situation of Corollary~\ref{cor:sit}, 
suppose that each 
$h_{U_i^{\n}} \to h|_{U_i}$ 
is an isomorphism. 
Then the isomorphism 
of $U_i^{\n}$ and $U_{j}^{\n}$
over $U_{ij}$ is canonically 
determined by $h$, 
so the quasi NC structure 
$\{U_i\}_{i\in \mathbb{I}}$ gives an NC structure
on $M$. 
\end{rmk}
We next describe
some 
condition for the functor 
$h_Y \to h$ to give an NC hull. 
For a functor $h \colon \nN \to \sS et$
and a
Cartesian 
diagram of algebras in $\nN$
\begin{align*}
\xymatrix{
\Lambda_{12} \ar[r]^{q_1} \ar[d]_{q_2} & \Lambda_1 \ar[d]^{p_1} \\
\Lambda_2 \ar[r]_{p_2} & \Lambda
}
\end{align*}
we consider the natural map 
\begin{align}\label{h:natu}
h(\Lambda_{12}) \to h(\Lambda_{1}) \times_{h(\Lambda)} h(\Lambda_2). 
\end{align}
\begin{prop}\label{prop:hull}
Let $h \colon \nN \to \sS et$ be a functor, 
$Y$ a smooth NC scheme and $\phi \colon h_Y \to h$
a natural transform. 
Suppose that $\phi$ is an
isomorphism on $\cC om$, and the following conditions hold
on the map (\ref{h:natu}):
\begin{enumerate}
\item
The map (\ref{h:natu})
is
surjective if  
 $\Lambda_1=\Lambda_2$ and $p_1=p_2$ is a central 
extension.  
\item
The map (\ref{h:natu}) is
bijective if $\Lambda$ is commutative and 
$\Lambda_2=\Lambda \oplus J$
for a $\Lambda$-module $J$. 
Here $\Lambda \oplus J$ 
is the trivial extension, i.e. 
$(a_1, m_1)(a_2, m_2)=(a_1 a_2, m_1 a_2+a_1 m_2)$.  
\end{enumerate}
Then $\phi$ is formally smooth, i.e. 
$\phi \colon h_Y \to h$ is an NC hull of $h$.
\end{prop}
\begin{proof}
The proof is similar to the classical argument 
on the existence of pro-representable 
hulls (cf.~\cite[Theorem~2.11]{Schle}).
We consider a central extension in $\nN$
\begin{align*}
0 \to J \to \Lambda' \stackrel{p}{\to} \Lambda \to 0
\end{align*}
and show that 
the map 
\begin{align}\label{maph}
h_Y(\Lambda') \to h(\Lambda') \times_{h(\Lambda)} h_Y(\Lambda)
\end{align}
is surjective. An element of the RHS of (\ref{maph})
is given by a commutative diagram
\begin{align}\label{diagram}
\xymatrix{
h_Y \ar[r]^{\phi} & h \\
\Spf \Lambda \ar[r] \ar[u]^{\alpha} & \Spf \Lambda'. \ar[u]_{\beta}
 \ar@{.>}[lu]^{\gamma}
}
\end{align}
We need to find a dotted arrow $\gamma$
so that the 
both of lower and upper triangles in (\ref{diagram})
are commutative. 
Since $h_Y$ is smooth, 
there is a dotted arrow $\gamma$
in (\ref{diagram}) 
so that the lower triangle is commutative. 
The set of possible choices of 
such a $\gamma$ is a principal homogeneous 
space over 
$\mathrm{Der}(\oO_{Y^{ab}}, J)$. 
Here 
we regard $J$ as an $\oO_{Y^{ab}}$-module 
by the algebra homomorphism $\alpha^{ab} \colon \oO_Y^{ab} \to \Lambda^{ab}$
induced by $\alpha$. 

On the other hand, 
we have the isomorphism
\begin{align*}
\Lambda' \times_{\Lambda} \Lambda'
\stackrel{\cong}{\to}\Lambda' \times_{\Lambda^{ab}}(\Lambda^{ab} \oplus J)
\end{align*}
given by 
$(x, y) \mapsto (x, x^{ab}+y-x)$. 
Therefore the assumption on the map (\ref{h:natu}) 
implies 
the surjection
\begin{align}\label{surject}
h(\Lambda') \times_{h(\Lambda^{ab})} h(\Lambda^{ab}\oplus J) 
\twoheadrightarrow
h(\Lambda') \times_{h(\Lambda)} h(\Lambda'). 
\end{align}
Let $\eta \in h(\Lambda)$
be the element corresponding 
to the composition $\phi \circ \alpha$
in the diagram (\ref{diagram}). 
The surjection (\ref{surject}) restricts to the surjection
\begin{align}\label{restrict}
h(p)^{-1}(\eta) \times h(q)^{-1}(\eta^{ab}) \twoheadrightarrow
h(p)^{-1}(\eta) \times h(p)^{-1}(\eta). 
\end{align}
Here $q \colon \Lambda^{ab} \oplus J \to \Lambda^{ab}$ is the projection. 
Since $\eta^{ab} \in h(\Lambda^{ab})$
is identified with the morphism $\alpha^{ab} \colon 
\oO_{Y}^{ab} \to \Lambda^{ab}$, 
we have
\begin{align*}
h(q)^{-1}(\eta^{ab})=
\mathrm{Der}(\oO_{Y^{ab}}, J).
\end{align*}
Therefore
the surjection (\ref{restrict})
shows that 
$\mathrm{Der}(\oO_{Y^{ab}}, J)$
acts on 
$h(p)^{-1}(\eta)$
transitively. 

Note that if $\gamma$ is a dotted arrow in (\ref{diagram})
so that the lower triangle is commutative, 
then we have
\begin{align*}
\phi \circ \gamma \in h(p)^{-1}(\eta), \ 
\beta \in h(p)^{-1}(\eta).
\end{align*}
Therefore by acting $\mathrm{Der}(\oO_{Y^{ab}}, J)$
to 
$\gamma$, 
we can also make the upper triangle of (\ref{diagram}) commutative. 
\end{proof}

Finally we state the 
construction of NC structures via functors. 
In the following 
proposition, we put $\nN_{\infty}=\nN$.

\begin{prop}\label{prop:hull2}
For $d \in [0, \infty]$, let 
$h \colon \nN_d \to \sS et$ be a functor, 
$Y$ a $d$-smooth NC scheme and $\phi \colon h_Y|_{\nN_d} \to h$
a natural transform. 
Suppose that $\phi$ is an
isomorphism on $\nN_e$
for some $e<d$, and the following conditions hold
on the map (\ref{h:natu})
\begin{enumerate}
\item
The map (\ref{h:natu})
is
a bijection if  
 $\Lambda_1=\Lambda_2$ and $p_1=p_2$ is a central 
extension with $\Lambda \in \nN_i$, $e\le i<d$.  
\item
The map (\ref{h:natu}) is
a bijection if $\Lambda$ is commutative and 
$\Lambda_2=\Lambda \oplus J$
for a $\Lambda$-module $J$. 
\end{enumerate}
Then $\phi$ is an isomorphism of functors. 
\end{prop}
\begin{proof}
The result for the $e=0$
case is given in~\cite[Lemma~2.3.6]{Kap17}. 
The $e>0$ case is similarly proved without any 
major modification. 
\end{proof}

\subsection{NC hull and pro-representable hull}\label{subsec:pro}
Let 
\begin{align*}
\nN^{\mathrm{loc}} \subset \nN
\end{align*}
be the subcategory of local 
finite dimensional $\mathbb{C}$-algebras, i.e. 
an object of $\nN^{\mathrm{loc}}$ is  
a finite dimensional $\mathbb{C}$-algebra $\Lambda$
having the unique two sided maximal ideal 
${\bf n} \subset \Lambda$. 
Note that a complete local $\mathbb{C}$-algebra $\widehat{R}$
defines the functor 
\begin{align*}
h_{\widehat{R}} \colon \nN^{\rm{loc}} \to \sS et
\end{align*}
by sending 
$\Lambda$
to the set of local $\mathbb{C}$-algebra 
homomorphisms 
$\widehat{R} \to \Lambda$.
We recall the notion of a pro-representable hull in~\cite{Schle}. 
\begin{defi}
Let $h^{\rm{loc}} \colon \nN^{\mathrm{loc}} \to \sS et$
be a functor. 
A pro-representable hull of $h^{\rm{loc}}$ 
is a complete local $\mathbb{C}$-algebra 
$\widehat{R}$ together with 
a formally smooth natural transform 
$h_{\widehat{R}} \to h^{\rm{loc}}$, 
which is bijective on 
$\mathbb{C}[t]/t^2$. 
\end{defi}
The proof similar to Lemma~\ref{NCh}
shows that 
a pro-representable hull 
is unique up
to non-canonical isomorphisms (cf.~\cite[Proposition~2.9]{Schle}). 

Let $h \colon \nN \to \sS et$
be a functor, and suppose that 
it has an NC hull 
$h_Y \to h$ for an NC scheme $Y$. 
Note that 
$h_Y(\mathbb{C}) \to h(\mathbb{C})$
is bijective, 
and they are identified with 
the set of closed points in $Y^{ab}$. 
Given a closed point $y \in Y^{ab}$, we
define the functor
\begin{align}\label{loc:funct}
h_{y}^{\rm{loc}} \colon \nN^{\rm{loc}} \to \sS et
\end{align}
by sending $(\Lambda, {\bf n})$
to the 
preimage of $y \in h(\mathbb{C})$
under the map 
\begin{align*}
h(\Lambda) \to h(\Lambda/{\bf n})=h(\mathbb{C}).
\end{align*} 
Let $m_y^{ab} \subset \oO_{Y^{ab}}$ be the 
ideal sheaf which defines $y$, 
and 
$m_y \subset \oO_Y$
the two sided ideal sheaf
given by the pull-back of $m_y^{ab}$
by the surjection 
$\oO_Y \twoheadrightarrow \oO_{Y^{ab}}$. 
We denote by
$\widehat{\oO}_{Y, y}$
the completion of $\oO_Y$
by $m_y$. 
We also set
\begin{align*}
\aA rt^{\rm{loc}} \cneq \nN^{\rm{loc}} \cap \cC om
\end{align*}
i.e. $\aA rt^{\rm{loc}}$ is the category of 
local Artinian $\mathbb{C}$-algebras. 
\begin{lem}\label{lem:prohull}
In the above situation, 
the functor $h_{y}^{\rm{loc}}$ has a 
pro-representable hull 
$h_{\widehat{\oO}_{Y, y}} \to h_{y}^{\rm{loc}}$, 
which is an isomorphism on $\aA rt^{\rm{loc}}$. 
\end{lem}
\begin{proof}
The NC hull 
$h_Y \to h$ induces the natural 
transform $h_{Y, y}^{\rm{loc}} \to h_y^{\rm{loc}}$, 
which is formally smooth and an
isomorphism on $\aA rt^{\rm{loc}}$
by the definition of NC hull. 
For $(\Lambda, {\bf n}) \in \nN^{\rm{loc}}$, 
giving a morphism $\Spf \Lambda \to Y$
sending $\Spec \Lambda/{\bf n}$ to $y$
is equivalent to giving 
a local $\mathbb{C}$-algebra homomorphism 
$\widehat{\oO}_{Y, y} \to \Lambda$. 
Hence 
$h^{\rm{loc}}_{Y, y}=h_{\widehat{\oO}_{Y, y}}$, 
and $h_{\widehat{\oO}_{Y, y}} \to h_{y}^{\rm{loc}}$
is a pro-representable hull. 
\end{proof}

\section{NC thickening of moduli spaces of quiver representations}\label{sec:NC2}
In this section, we construct quasi NC structures on 
the moduli spaces of representations of quivers with
relations. 
\subsection{Representations of quivers}\label{subsec:Q}
Recall that a quiver 
consists of  
data
\begin{align*}
Q=(Q_0, Q_1, h, t)
\end{align*}
where $Q_0$, $Q_1$ are finite
sets (called the set of \textit{vertices}, \textit{arrows} respectively)
and 
$h, t \colon Q_1 \to Q_0$
are maps. 
The maps $h, t$ indicate the vertices at the head, tail
of each arrow respectively. 
\begin{defi}
A representation of a quiver $Q$
over an NC scheme $(Y, \oO_Y)$
consists of 
data
\begin{align}\label{dataE}
\wW=(\{\wW_v\}_{v \in Q_0}, 
\{\phi_a\}_{a \in Q_1})
\end{align}
where each 
$\wW_v$ is 
an object of $\Coh(Y)$, and 
$\phi_a \colon \wW_{t(a)} \to \wW_{h(a)}$
is a
morphism of coherent left $\oO_Y$-modules. 
\end{defi}
Given a representation (\ref{dataE}) of $Q$ over $Y$, 
we set
\begin{align*}
\wW_{\bullet} \cneq \bigoplus_{v\in Q_0} \wW_v. 
\end{align*}
Let $\mathbb{C}[Q]$ be 
the path algebra of $Q$. 
By the definition, we have the natural 
algebra homomorphism
\begin{align}\label{natural}
\mathbb{C}[Q]
\to \Hom_{\oO_Y}(\wW_{\bullet}, \wW_{\bullet})
\end{align}
sending $a \in Q_1$ to $\phi_a$.  
Recall that a
quiver with relation 
is a pair $(Q, I)$, 
where $Q$ is a quiver 
and $I \subset \mathbb{C}[Q]$
is a two sided ideal. 
\begin{defi}
Let $(Q, I)$ be a quiver with 
relation. 
A representation of $(Q, I)$ over an NC scheme 
$Y$
is a representation of $Q$ over $Y$
such that the map (\ref{natural}) 
is zero on $I$. 
\end{defi}
For 
another representation 
$\wW'$
of $(Q, I)$
over $Y$, 
the set of morphisms 
$\Hom(\wW, \wW')$ 
consists of
coherent left $\oO_Y$-module 
homomorphisms $\wW_v \to \wW_v'$ for each 
$v \in Q_0$ which commute with 
$\phi_a$ and $\phi_a'$.
The category
\begin{align}\label{QI}
\mathrm{Rep}((Q, I)/Y)
\end{align}
is defined to be the category of 
representations of $(Q, I)$ over $Y$.  
For objects $\wW, \wW'$ in (\ref{QI}), we 
also have the 
sheaf homomorphisms
$\hH om(\wW, \wW')$
on $Y$
by associating an open subset
$U \subset Y$ with 
$\Hom(\wW|_{U}, \wW'|_{U})$, 
which is an object of $\Coh(Y)$
if $Y$ is a commutative scheme. 
If $I=\{0\}$, or $Y=\Spf \Lambda$ for 
$\Lambda \in \nN$,
we simply write (\ref{QI})
as $\mathrm{Rep}(Q/Y)$, or 
$\mathrm{Rep}((Q, I)/\Lambda)$
respectively. 
If further $Y=\Spec \mathbb{C}$, we
write (\ref{QI}) as $\mathrm{Rep}(Q, I)$, 
and $\mathrm{Rep}(Q)$ if $I=\{0\}$.
\begin{rmk}
For $\Lambda \in \nN$, 
the category 
$\mathrm{Rep}((Q, I)/\Lambda)$
is equivalent to the category 
of collections of finitely 
generated left $\Lambda$-modules 
$\{W_v\}_{v\in Q_0}$ together with 
left $\Lambda$-module homomorphisms 
$\phi_a \colon W_{t(a)} \to W_{h(a)}$,
such that the natural 
map $I \to \Hom_{\Lambda}(W_{\bullet}, W_{\bullet})$
is zero. 
We call such a collection $\{W_v\}_{v\in Q_0}$ as a
representation of $(Q, I)$ over $\Lambda$. 
\end{rmk}
Let $Q$ be a quiver and 
$\wW$ a representation of it over an
NC scheme $Y$. 
We call it 
\textit{flat} if 
each $\wW_v$ is a flat left $\oO_Y$-module. 
This is equivalent to that each $\wW_v$ is a 
locally free left $\oO_Y$-module.
We prepare the following lemma: 
\begin{lem}\label{lem:I}
(i) Let $\wW$ be a flat
 representation of $(Q, I)$ over a commutative scheme $T$, 
and $\jJ$ a coherent sheaf on $T$.
Suppose that for any $t\in T$, 
we have $\Hom(\wW|_{t}, \wW|_{t})=\mathbb{C}$.  
Then the morphism 
\begin{align*}
\jJ \to \hH om(\wW, \jJ \otimes_{\oO_T}\wW)
\end{align*}
given by $u\mapsto u\otimes \id$ is an isomorphism. 

(ii) Let $0 \to J \to \Lambda' \to \Lambda \to 0$
be a central extension in $\nN$ and 
$\wW$ a flat representation of $(Q, I)$ over $\Lambda'$. 
Then any automorphism of $\wW$ which is identity over
$\Lambda$ is given by the left multiplication
of a central element $1+u$ for $u\in J$. 
\end{lem}
\begin{proof}
(i)
We may assume that $I=\{0\}$. 
 For $t\in T$, 
we have the isomorphism
in $D(\mathrm{Rep}(Q))$: 
\begin{align*}
\oO_{t} \dotimes_{\oO_T}
\dR \hH om(\wW, \wW) \cong 
\dR \Hom(\wW|_{t}, \wW|_{t}). 
\end{align*}
Hence we have the spectral sequence
\begin{align*}
E_{2}^{p, q}=
\tT or_{-p}^{\oO_T}(\oO_t, \eE xt^q(\wW, \wW))
\Rightarrow \Ext^{p+q}(\wW|_{t}, \wW|_{t}).
\end{align*}
Since $\mathbb{C}[Q]$ is hereditary, 
we have 
$\Ext^{\ge 2}(\wW|_{t}, \wW|_{t})=0$.  
Using the above spectral sequence 
and the assumption $\Hom(\wW|_{t}, \wW|_{t})=\mathbb{C}$,
we see that  
\begin{align*}
\eE xt^{\ge 2}(\wW, \wW)=0, \ 
E_{2}^{-1, 1}=0, \ 
E_2^{0, 0}=\mathbb{C}.
\end{align*}
The vanishing of $E_{2}^{-1, 1}$
shows that 
$\eE xt^1(\wW, \wW)$ is flat over $\oO_T$, 
and 
$E_2^{0, 0}=\mathbb{C}$
shows that the natural map 
$\oO_T \to \hH om(\wW, \wW)$
is an isomorphism\footnote{Here we use the 
assumption that $T$ is commutative, 
as otherwise there is no natural map
$\oO_T \to \hH om(\wW, \wW)$}. 
In particular for $\jJ \in \Coh(T)$, 
we have the isomorphism
\begin{align*}
\hH^0(\jJ \dotimes_{\oO_T} \dR \hH om(\wW, \wW))
\cong \jJ. 
\end{align*}
Then the result follows by 
taking the zero-th cohomology of the following
isomorphism in $D(\Coh(T))$
\begin{align*}
\jJ \dotimes_{\oO_T} \dR \hH om(\wW, \wW)
\cong \dR \hH om(\wW, \jJ \dotimes_{\oO_T} \wW). 
\end{align*}

(ii) 
Let $g$ be an automorphism of $\wW$
which is identity over $\Lambda$. 
We have the commutative diagram
\begin{align*}
\xymatrix{
0 \ar[r] & J\otimes_{\Lambda^{ab}} \wW^{ab}
\ar[r]^-{\iota} \ar[d]^{\id} & \wW \ar[r]^-{p}
\ar[d]^{g}
 & \Lambda \otimes_{\Lambda'} \wW \ar[r]\ar[d]^{\id} & 0 \\
0 \ar[r] & J\otimes_{\Lambda^{ab}} \wW^{ab} 
\ar[r]^-{\iota} & 
\wW \ar[r]^-{p}
 & \Lambda \otimes_{\Lambda'}\wW
 \ar[r] & 0. 
}
\end{align*}
By the above commutative diagram, 
the isomorphism $g$ is written 
as $\id +\iota \circ \alpha \circ p$
for some morphism
$\alpha \colon \Lambda \otimes_{\Lambda'} \wW \to 
J \otimes_{\Lambda^{ab}} \wW^{ab}$.  
The morphism $\alpha$ descends to the morphism
$\wW^{ab} \to J \otimes_{\Lambda^{ab}} \wW^{ab}$, 
which is given by $u \otimes \id$ for some $u\in J$
by (i). Therefore $g$ is the left multiplication of $1+u$. 
\end{proof}

Let $\wW$ be a representation of $Q$ over 
an NC scheme $Y$. 
We consider the natural
morphism of sheaves of $\oO_Y$ bi-modules
\begin{align}\notag
\wW_{\bullet} \otimes_{\mathbb{C}}
\Hom_{\oO_Y}(\wW_{\bullet}, \wW_{\bullet}) \otimes_{\mathbb{C}}
\hH om_{\oO_Y}(\wW_{\bullet}, \oO_Y) \to \oO_Y
\end{align}
given by 
$e \otimes f \otimes g \mapsto g\circ f(e)$. 
By composing it with (\ref{natural})
and the inclusion $I \subset \mathbb{C}[Q]$, 
we obtain the morphism of sheaves of $\oO_Y$ bi-modules
\begin{align}\label{nmor}
\wW_{\bullet} \otimes_{\mathbb{C}}
I \otimes_{\mathbb{C}}
\hH om_{\oO_Y}(\wW_{\bullet}, \oO_Y) \to \oO_Y. 
\end{align}
\begin{defi}\label{defi:ideal}
In the above situation, 
we define 
the \textit{ideal of relations in}
$I$ 
to be the image of (\ref{nmor}), and denote 
it by $\jJ_{I} \subset \oO_Y$. 
\end{defi}
If $\wW$ is flat, then 
$\jJ_I$ is locally the two sided ideal 
generated by the matrix components 
of the morphism (\ref{natural})
restricted to $I$. 
In particular, the map (\ref{natural}) is zero on $I$ 
if and only if $\jJ_I=0$.
If
$Z$ is the closed NC subscheme of $Y$
defined by 
$\oO_Z \cneq \oO_Y/\jJ_I$, 
the representation of $Q$
\begin{align*}
\wW|_{Z} \cneq 
(\{\oO_Z \otimes_{\oO_Y}\wW_{v}\}_{v\in Q_0}, 
\{\id \otimes \phi_a\}_{a\in Q_1})
\end{align*} 
over $Z$ is a representation of $(Q, I)$. 
\subsection{Moduli spaces of representations of quivers}
Let $(Q, I)$ be a quiver with relation. 
We set $\Gamma_{Q} \cneq \mathbb{Z}^{Q_0}$, and 
its inner product by 
\begin{align*}
\gamma \cdot \gamma'=\sum_{v\in Q_0} \gamma_v \cdot \gamma_v'. 
\end{align*} 
For an object $W$ in $\mathrm{Rep}(Q)$, its dimension 
vector $\dim W$ is defined by 
\begin{align*}
\dim W \cneq (\dim W_v)_{v \in Q_0} \in \Gamma_{Q}. 
\end{align*}
We recall the notion of King's 
$\theta$-stability on $\mathrm{Rep}(Q, I)$:
\begin{defi}\emph{(\cite{Kin})}
For $\theta \in \Gamma_{Q}$, 
a representation $W$ of $(Q, I)$
is called $\theta$-(semi)stable 
if $\theta \cdot \dim W=0$, and 
for any 
subobject $0 \neq W' \subsetneq W$
in $\mathrm{Rep}(Q, I)$, 
we have the inequality
\begin{align*}
\theta \cdot \dim W' > (\ge) 0. 
\end{align*}
\end{defi}
For $\gamma, \theta \in \Gamma_{Q}$ with 
$\theta \cdot \gamma=0$, the 2-functor
\begin{align}\label{stack:Q}
\mathfrak{M}^{}_{Q, I, \theta}(\gamma)
\colon \sS ch/\mathbb{C} \to \gG roupoid
\end{align}
is defined by sending 
a $\mathbb{C}$-scheme $T$
to the groupoid of flat
representations $\wW$ of $(Q, I)$ over $T$
such that 
$\wW|_{t}$ for any $t \in T$ 
is a 
$\theta$-semistable representation of $(Q, I)$
with dimension vector $\gamma$.
The 2-functor (\ref{stack:Q}) is known to be an algebraic 
stack of finite type over $\mathbb{C}$. 
Indeed, let $W$ be the affine space given by
\begin{align*}
W \cneq \prod_{a \in Q_1} \Hom(W_{t(a)}, W_{h(a)}). 
\end{align*}
By the construction of $W$, 
there exists a tautological 
representation of $Q$ over $W$. 
Let  
$M \subset W$
be the subscheme 
defined by the ideal of relations in $I$, 
whose  
 closed points
correspond to $(Q, I)$-representations. 
The subset of points $M^{ss} \subset M$ corresponding to 
$\theta$-semistable representations form 
an open subset of $M$. 
Also the group 
$G \cneq \prod_{v \in Q_0} \GL(W_v)$
acts on $W$ 
by 
\begin{align*}
(g_v) \cdot (\phi_a) =(g_{h(a)}^{-1} \circ \phi_a \circ g_{t(a)}).
\end{align*}
The $G$-action on $W$ preserves $M^{ss}$, 
and the stack (\ref{stack:Q}) is given by the quotient stack
\begin{align*}
\mathfrak{M}_{Q, I, \theta}^{}(\gamma)
=\left[M^{ss}/G \right]. 
\end{align*}
On the other hand, let 
\begin{align}\label{funct:M}
\mM_{Q, I, \theta}(\gamma)
\colon \sS ch/\mathbb{C} \to \sS et
\end{align}
be the functor 
defined by 
sending a 
$\mathbb{C}$-scheme $T$
to the set of equivalence classes of 
$\theta$-stable flat 
representations of $(Q, I)$ over 
$T$.
Here
 $\wW$ and $\wW'$ are called \textit{equivalent} if 
there is a line bundle $\lL$ on $T$
such that $\wW$ and $\wW'\otimes \lL$ are isomorphic 
in $\mathrm{Rep}((Q, I)/T)$. 
The following result was proved by King. 
\begin{thm}\emph{(\cite{Kin})}\label{thm:Kin}
If $\gamma \in \Gamma_{Q}$ is primitive, the 
functor $\mM_{Q, I, \theta}(\gamma)$ is represented by 
a quasi-projective scheme
$M_{Q, I, \theta}(\gamma)$, which 
is projective if $Q$ does not contain a loop. 
If $I=\{0\}$, the 
moduli space $M_{Q, \theta}(\gamma) \cneq M_{Q, \{0\}, \theta}(\gamma)$
is non-singular.   
\end{thm}
Indeed 
let $M^{s} \subset M^{ss}$ be
the open subset consisting of 
$\theta$-stable representations. 
For 
$w \in M^{s}$,  
the subgroup of $G$
which fixes $w$ coincides with the diagonal subgroup
$\mathbb{C}^{\ast} \subset G$. 
Hence the group $\overline{G} \cneq G/\mathbb{C}^{\ast}$ acts 
on $M^{s}$ without fixed points, and $M_{Q, I, \theta}(\gamma)$ is given by
\begin{align}\label{moduli:Q}
M_{Q, I, \theta}(\gamma)=M^{s}/\overline{G}. 
\end{align}
In particular, 
the open 
substack of
$\mathfrak{M}_{Q, I, \theta}(\gamma)$
consisting of stable representations
is a $\mathbb{C}^{\ast}$-gerbe over (\ref{moduli:Q}).
By the construction, we have the closed embedding
\begin{align*}
M_{Q, I, \theta}(\gamma)
\hookrightarrow 
M_{Q, \theta}(\gamma)
\end{align*}
such that $M_{Q, I, \theta}(\gamma)$
is defined by the 
ideal of relations in $I$
on the smooth moduli space
$M_{Q, \theta}(\gamma)$. 
Since $M_{Q, I, \theta}(\gamma)$ 
represents the functor (\ref{funct:M}), 
there exists a 
universal $(Q, I)$-representation,
i.e. 
a family of representations of $(Q, I)$
\begin{align}\label{univ}
\vV=(\{\vV_v\}_{v\in Q_0}, \{\phi_a\}_{a\in Q_1})
\end{align}
over $M_{Q, I, \theta}(\gamma)$ 
such that the map $f \mapsto f^{\ast}\vV$
 gives the functorial isomorphism
\begin{align*}
\Hom(T, M_{Q, I, \theta}(\gamma)) \stackrel{\cong}{\to}
 \mM_{Q, I, \theta}(\gamma)(T)
\end{align*}
for any $\mathbb{C}$-scheme $T$.

\subsection{Construction of NC hull}
Let $\gamma \in \Gamma_{Q}$ be a primitive element. 
We consider the smooth moduli space 
$M_{Q, \theta}(\gamma)$
of representations of $Q$ without relation 
(cf.~Theorem~\ref{thm:Kin}), 
and a universal representation 
$\vV$ on it
given by (\ref{univ}) for $I=\{0\}$. 
We take an affine open subset
\begin{align*}
U \subset M_{Q, \theta}(\gamma)
\end{align*}
such that 
each $\vV_v|_{U}$ is isomorphic to 
$\oO_{U} \otimes_{\mathbb{C}} W_v$, 
where $W_v$ is  
a $\mathbb{C}$-vector space with 
dimension $\gamma_v$. 
Since
$U$ is smooth, 
it admits an
NC smooth thickening $U^{\n}=(U, \oO_{U}^{\rm{nc}})$, 
 which 
is unique up to non-canonical isomorphisms (cf.~\cite[Theorem~1.6.1]{Kap17}). 
We set
\begin{align*}
\vV^{\rm{nc}}_{U, v} \cneq \oO_{U}^{\rm{nc}} \otimes_{\mathbb{C}} W_v, \
v\in Q_0
\end{align*}
which we regard as left $\oO_{U}^{\n}$-modules. 
Since $\oO_{U}^{\n} \twoheadrightarrow \oO_{U}$ is surjective, 
we can lift each universal 
morphism $\phi_a \colon \vV_{t(a)} \to \vV_{h(a)}$
restricted to $U$ to a left
$\oO_{U}^{\n}$-module homomorphism
\begin{align}\label{data}
\phi_{U, a}^{\n} \colon \vV_{U, t(a)}^{\n} \to \vV_{U, h(a)}^{\n}, \
a \in Q_1.
\end{align}
Then the data
\begin{align}\label{repU}
\vV_{U}^{\n} \cneq (\{\vV_{U, v}^{\n}\}_{v\in Q_0}, 
\{\phi_{U, a}^{\n}\}_{a\in Q_1})
\end{align}
is a representation of $Q$ over $U^{\n}$. 

We 
define the functor
\begin{align}\label{fun:h}
h_{Q, \theta}(\gamma) \colon \nN \to \sS et
\end{align}
by sending  
$\Lambda \in \nN$
to the isomorphism classes of 
triples
$(f, \wW, \psi)$:
\begin{itemize}
\item $f$ is a morphism 
$\Spec \Lambda^{ab} \to M_{Q, \theta}(\gamma)$
of schemes. 
\item $\wW$ is a flat representation of $Q$ over 
$\Lambda$. 
\item $\psi$ is an isomorphism 
$\psi \colon \wW^{ab}
\stackrel{\cong}{\to}f^{\ast}\vV$
as representations of $Q$ over 
$\Lambda^{ab}$. 
\end{itemize}
An isomorphism 
$(f, \wW, \psi) \to (f', \wW', \psi')$ 
exists if $f=f'$, and 
there is an isomorphism $\wW \to \wW'$ 
as representations of $Q$ over $\Lambda$ 
commuting $\psi$, $\psi'$.
Note that we have
\begin{align}\label{note:hQ}
h_{Q, \theta}(\gamma)|_{\cC om}=h_{M_{Q, \theta}(\gamma)}
\end{align}
\begin{prop}\label{prop:natural}
The natural transformation
\begin{align}\label{hUd}
h_{U^{\n}} \to h_{Q, \theta}(\gamma)|_{U}
\end{align}
sending $g \colon \Spf \Lambda \to U^{\n}$
to $(g^{ab}, g^{\ast} \vV_U^{\n}, \id)$ 
is an NC hull of $h_{Q, \theta}(\gamma)|_{U}$.   
\end{prop}
\begin{proof}
We write $h=h_{Q, \theta}(\gamma)|_{U}$ for simplicity. 
Since $h_{U^{\n}}|_{\cC om}=h_{U}$, 
the natural transformation (\ref{hUd})
is an isomorphism
on $\cC om$
by (\ref{note:hQ}). 
Therefore by 
Proposition~\ref{prop:hull},  
it is enough to show the following: 
for surjections
$p_1 \colon \Lambda_1 \twoheadrightarrow \Lambda$, 
$p_2 \colon \Lambda_2 \twoheadrightarrow \Lambda$  
in $\nN$, 
$\Lambda_{12} \cneq \Lambda_1 \times_{\Lambda} \Lambda_2$, 
the natural 
map
\begin{align}\label{h12}
h(\Lambda_{12})
\to h(\Lambda_1) \times_{h(\Lambda)} h(\Lambda_2)
\end{align}
is surjective. 
The RHS of (\ref{h12}) consists of 
triples 
\begin{align}\label{RHS}
(f_j, \wW_j, \psi_j) \in h(\Lambda_j), \ 
j=1, 2, 
\end{align}
which are 
isomorphic over $\Lambda$, 
i.e. $f\cneq f_1|_{\Spec \Lambda^{ab}}=f_2|_{\Spec \Lambda^{ab}}$
and
there is an isomorphism 
of representations of $Q$ over $\Lambda$
\begin{align}\label{isom:h}
\gamma \colon \Lambda \otimes_{\Lambda_1}
 \wW_1 \stackrel{\cong}{\to} \Lambda \otimes_{\Lambda_2} \wW_2
\end{align}
which commutes with
$\Lambda^{ab} \otimes_{\Lambda_j^{ab}}\psi_j$
for $j=1, 2$. 
 Let us write 
$\wW_j$ as 
\begin{align*}
\wW_j=(\{W_{j, v}\}_{v\in Q_0}, \{\phi_{j, a}\}_{a \in Q_1})
\end{align*}
as representations of $Q$ over $\Lambda_j$. 
Then the isomorphism $\gamma$ consists of collections of 
isomorphisms
of left $\Lambda$-modules
\begin{align*}
\gamma_{v} \colon \Lambda \otimes_{\Lambda_1}
W_{1, v} \stackrel{\cong}{\to} \Lambda \otimes_{\Lambda_2}
W_{2, v}
\end{align*} 
for each $v \in Q_1$
which commute with $\Lambda \otimes_{\Lambda_{j}} \phi_{j, a}$.  
We set
\begin{align*}
W_{12, v} \cneq
\{ (x, y) \in W_{1, v} \times W_{2, v} : 
\gamma_v \circ (1 \otimes x)=1 \otimes y\}. 
\end{align*}
Then $W_{12, v}$
is a projective left $\Lambda_{12}$-module 
by~\cite[Theorem~2.1]{Mil}.
Therefore the data
\begin{align}\label{Mil}
\wW_{12} \cneq 
\left(\{W_{12, v}\}_{v\in Q_0}, 
\{\phi_{1, a} \times \phi_{2, a} \}_{a\in Q_1} \right)
\end{align} 
determines a flat
representation of $Q$ over 
$\Lambda_{12}$.  
Also 
since $\Lambda_{12}^{ab}=\Lambda_1^{ab} \times_{\Lambda^{ab}}
\Lambda_2^{ab}$, 
the morphisms $f_1$, $f_2$
and $f$
induce the 
morphism of schemes 
\begin{align*}
f_{12} \cneq f_1 \times_f f_2 \colon 
\Spec \Lambda_{12}^{ab} \to U.
\end{align*}
Finally
since $\gamma$ commutes with 
$\Lambda^{ab} \otimes_{\Lambda_j^{ab}}\psi_j$, 
we have the isomorphism
\begin{align*}
\psi_{12} \cneq 
\psi_1 \times \psi_2 \colon 
\wW_{12}^{ab} \stackrel{\cong}{\to}
f_{12}^{\ast}\vV 
\end{align*}
of representations of $Q$ over $\Lambda_{12}^{ab}$. 
Therefore the triple
$(f_{12}, \wW_{12}, \psi_{12})$ determines an element of 
the LHS of (\ref{h12}), 
which is mapped to the triples (\ref{RHS})
by the map (\ref{h12}). 
Therefore the map (\ref{h12}) is surjective. 
\end{proof}
\begin{rmk}\label{rmk:aut}
Note that the
resulting element
 $(f_{12}, \wW_{12}, \psi_{12}) \in h(\Lambda_{12})$
may also depend on a choice of $\gamma$. 
If it really depends on $\gamma$, the 
map (\ref{h12})
is not bijective, and (\ref{hUd}) 
is not isomorphic. 
In order to show the independence of $\gamma$, 
one needs to show 
that any automorphism
of $\Lambda \otimes_{\Lambda_j} \wW_j$
extends to that of $\wW_j$, 
which might not be true if
$\Lambda_j$ is non-commutative. 
A similar issue
occurs in the proof of~\cite[Proposition~5.4.3]{Kap17}, 
which caused its gap
as pointed out in~\cite[Remark~4.1.4]{PoTu}. 
\end{rmk}

\begin{rmk}
A priori, the representation 
$\vV_U^{\n}$ depends on
the choices of lifting (\ref{data}).
However Proposition~\ref{prop:natural}
implies that
different choices of lifting (\ref{data})
yield isomorphic representations
after pulling back by some 
automorphism of $U^{\n}$. 
See Corollary~\ref{Vij}.  
\end{rmk}

We next consider the moduli 
space $M_{Q, I, \theta}(\gamma)$ 
in Theorem~\ref{thm:Kin} 
for 
non-zero relation 
$I \subset \mathbb{C}[Q]$. 
We embed it into the smooth 
moduli space 
\begin{align*}
M_{Q, I, \theta}(\gamma) \subset 
M_{Q, \theta}(\gamma)
\end{align*} 
and take an affine open subset 
$U \subset M_{Q, \theta}(\gamma)$
as before. 
Let $U^{\n}$ be an NC
smooth thickening of
$U$ and 
$\vV_U^{\n}$ 
a lift of a universal representation
(\ref{repU})
to $U^{\n}$. 
The representation $\vV_U^{\n}$ 
together with the relation $I$
determine the 
two sided ideal 
$\jJ_{I, U} \subset \oO_U^{\n}$ 
of relations in $I$ 
(cf.~Definition~\ref{defi:ideal}). 
We set  
\begin{align}\label{Vthick}
V \cneq M_{Q, I, \theta}(\gamma) \cap U, \ 
\oO_{V}^{\n} \cneq \oO_U^{\n}/\jJ_{I, U}, \ 
V^{\n} \cneq (V, \oO_V^{\n}).
\end{align}
Since $(\oO_V^{\n})^{ab}=\oO_V$, 
the affine NC scheme
$V^{\n}$ is a closed NC subscheme of $U^{\n}$. 
The restriction 
$\vV_V^{\n} \cneq \vV_U^{\n}|_{V^{\n}}$
is a representation of $(Q, I)$ over
$V^{\n}$. 

We also define the functor
\begin{align}\label{h:Nset}
h_{Q, I, \theta}(\gamma) \colon \nN \to \sS et
\end{align}
to be the sub functor of $h_{Q, \theta}(\gamma)$
in (\ref{fun:h}), 
sending $\Lambda \in \nN$ to the 
set of triples 
$(f, \wW, \psi) \in h_{Q, \theta}(\gamma)(\Lambda)$
such that
$f$ factors through 
$M_{Q, \theta, I}(\gamma)$
and $\wW$ is a representation of $(Q, I)$
over $\Lambda$.  
Note that we have
\begin{align*}
h_{Q, I, \theta}(\gamma)|_{\cC om}=h_{M_{Q, I, \theta}(\gamma)}.
\end{align*}
\begin{prop}\label{prop:natural2}
The natural transformation
\begin{align}\label{hUd2}
h_{V^{\n}} \to h_{Q, I, \theta}(\gamma)|_{V}
\end{align}
sending $g \colon \Spf \Lambda \to V^{\n}$
to $(g^{ab}, g^{\ast} \vV_V^{\n}, \id)$ 
is an NC hull of $h_{Q, I, \theta}(\gamma)|_{V}$.   
\end{prop}
\begin{proof}
By Proposition~\ref{prop:natural}, it 
is enough to note that the 
following diagram
of functors
\begin{align*}
\xymatrix{
h_{U^{\n}} \ar[r] & h_{Q, \theta}(\gamma)|_{U} \\
h_{V^{\n}} \ar[r] \ar@{^{(}->}[u]
 & h_{Q, I, \theta}(\gamma)|_{V} \ar@{^{(}->}[u]
}
\end{align*}
is Cartesian. 
Indeed for 
$g \colon \Spf \Lambda \to U^{\n}$, 
suppose that 
$g^{\ast}\vV_{U}^{\n}$ 
is a representation of $(Q, I)$ over $\Lambda$. 
This is equivalent to that the 
natural morphism
\begin{align*}
(\Lambda \otimes_{\oO_{U}^{\n}}\vV_{U}^{\n}) \otimes_{\mathbb{C}} I
 \otimes_{\mathbb{C}} \Hom_{\Lambda}(\Lambda \otimes_{\oO_{U}^{\n}}\vV_{U}^{\n}, \Lambda)
\to \Lambda
\end{align*}
is a zero map, 
which is equivalent to that 
\begin{align*}
\Lambda \otimes_{\oO_U^{\n}}
 \jJ_{I, U} \otimes_{\oO_U^{\n}} \Lambda \to \Lambda
\end{align*}
induced by $\jJ_{I, U} \subset \oO_U^{\n}$ is 
a zero map. 
The last condition 
is also equivalent to 
that $g^{\ast} \colon \oO_U^{\n} \to \Lambda$
vanishes on $\jJ_{I, U}$, hence 
$g$ factors through $V^{\n} \hookrightarrow U^{\n}$. 
\end{proof}
Let $\{U_i\}_{i\in \mathbb{I}}$ 
be a sufficiently small 
affine open covering of $M_{Q, \theta}(\gamma)$, 
and set $V_i=M_{Q, I, \theta}(\gamma) \cap U_i$. 
By applying the above construction 
to $U=U_i$, 
we
obtain the NC schemes
\begin{align}\label{UiVi}
U_i^{\n} \cneq (U_i, \oO_{U_i}^{\n}), \ 
V_i^{\n} \cneq (V_i, \oO_{V_i}^{\n})
\end{align} 
and 
the representations
$\vV_{U_i}^{\n}$
of $Q$
over $U_i^{\n}$, 
$\vV_{V_i}^{\n} \cneq \vV_{U_i}^{\n}|_{V_i}$ 
of $(Q, I)$
over 
$V_i^{\n}$
respectively. 
We have the following corollary: 
\begin{cor}\label{cor:NCV}
There exist isomorphisms
of NC schemes
\begin{align}\label{Vij}
\phi_{ij} \colon 
V_{j}^{\n}|_{V_{ij}} \stackrel{\cong}{\to}
V_i^{\n}|_{V_{ij}}
\end{align}
giving a quasi NC structure on 
$M_{Q, I, \theta}(\gamma)$, 
and isomorphisms 
of representations of $(Q, I)$: 
\begin{align}\label{isom:g}
g_{ij} \colon 
\phi_{ij}^{\ast}\vV_{V_i}^{\n}|_{V_{ij}} \stackrel{\cong}{\to}
\vV_{V_j}^{\n}|_{V_{ij}}.
\end{align} 
\end{cor} 
\begin{proof}
The existence of isomorphisms
(\ref{Vij}) follow from
Proposition~\ref{prop:natural2} and  
Corollary~\ref{cor:sit}. 
Since one can choose $\phi_{ij}$ 
commuting with
natural transforms
 $h_{V_i^{\n}} \to h_{Q, I, \theta}(\gamma)|_{V_i}$
in Proposition~\ref{prop:natural2}, 
we 
also
have isomorphisms
(\ref{isom:g}).
\end{proof}

\subsection{Comparison with the formal deformations of quiver representations}
\label{subsec:compare}
Here we use the notation in Subsection~\ref{subsec:pro}. 
For an object $W \in \mathrm{Rep}(Q, I)$, 
the formal non-commutative deformation 
functor 
\begin{align}\label{Def}
\mathrm{Def}_{W}^{\n} \colon 
\nN^{\rm{loc}} \to \sS et
\end{align}
is defined by sending $(\Lambda, {\bf n})$
to the set of isomorphism classes 
$(\wW, \psi)$, where $\wW$
is a flat representation of $(Q, I)$ over 
$\Lambda$, 
and 
$\psi \colon \Lambda/{\bf n} \otimes_{\Lambda} \wW 
\stackrel{\cong}{\to} W$ 
is an isomorphism 
as representations of $(Q, I)$. 
The non-commutative deformation 
functor (\ref{Def}) 
was studied by Laudal~\cite{Lau}, 
where the existence of a pro-representable hull
is proved. 
We also refer to~\cite{Erik}, \cite{ESe}, \cite{ELO}, \cite{ELO2}, \cite{ELO3}
for details on formal non-commutative deformation theory. 

Recall that  
the formal commutative 
deformation space of $W$
is given by
the solution of the
Mauer-Cartan equation 
of the dg-algebra
$\dR \Hom(W, W)$, up
to gauge equivalence. 
Let 
\begin{align}\label{Ainf}
(\Ext^{\ast}(W, W), \{m_n\}_{n\ge 2})
\end{align}
be the minimal $A_{\infty}$-algebra
which is $A_{\infty}$ quasi-isomorphic to 
$\dR \Hom(W, W)$. 
Here we take the
$\Ext$-groups in the category $\mathrm{Rep}(Q, I)$. 
By~\cite{ESe},
the pro-representable hull of 
(\ref{Def}) 
is described in terms of the $A_{\infty}$-algebra (\ref{Ainf}). 
Let
\begin{align*}
m_n \colon \Ext^1(W, W)^{\otimes n} \to \Ext^2(W, W)
\end{align*}
be the $n$-th $A_{\infty}$-product. 
Below for a vector space $V$, 
we denote by $\widehat{T}^{\bullet}(V)$ the
completed tensor algebra given by
\begin{align*}
\widehat{T}^{\bullet}(V)=\prod_{n\ge 0} V^{\otimes n}. 
\end{align*}
Let  
\begin{align*}
J_W \subset \widehat{T}^{\bullet}(\Ext^1(W, W)^{\vee})
\end{align*}
be the topological closure of 
the two sided ideal of the 
completed tensor algebra 
of $\Ext^1(W, W)^{\vee}$
generated by the image of the map
\begin{align*}
\sum_{n\ge 2} m_n^{\vee} 
\colon \Ext^2(W, W)^{\vee} \to \widehat{T}^{\bullet}(\Ext^1(W, W)^{\vee}). 
\end{align*}
The pro-representable hull of (\ref{Def}) is given 
by the quotient algebra
\begin{align*}
R_{W}^{\n} \cneq 
\widehat{T}^{\bullet}(\Ext^1(W, W)^{\vee})/J_W.  
\end{align*}

Let us take an open neighborhood 
$[W] \in V \subset M_{Q, I, \theta}(\gamma)$, 
and a non-commutative thickening 
$V^{\n}=(V, \oO_V^{\n})$
as in the previous subsection (\ref{Vthick}). 
\begin{lem}\label{lem:complete}
The completion 
$\widehat{\oO}_{V, [W]}^{\n}$ is 
isomorphic to $R_{W}^{\n}$.  
\end{lem}
\begin{proof} 
Let
\begin{align*}
h^{\rm{loc}}_{Q, I, \theta}(\gamma)_{[W]}
\colon \nN^{\rm{loc}} 
\to \sS et
\end{align*}
be the functor constructed from 
$h_{Q, I, \theta}(\gamma)$
by (\ref{loc:funct}), 
which has a pro-representable 
hull $h_{\widehat{\oO}_{V, [W]}^{\n}} \to 
h^{\rm{loc}}_{Q, I, \theta}(\gamma)_{[W]}$
by Lemma~\ref{lem:prohull}. 
We construct the natural transform
\begin{align}\label{isom:funct}
h^{\rm{loc}}_{Q, I, \theta}(\gamma)_{[W]} \to
\mathrm{Def}_{W}^{\n}
\end{align}
by sending 
triples $(f, \wW, \psi)$ of the $(\Lambda, {\bf n})$-point of 
the LHS of (\ref{isom:funct}) to 
the pair $(\wW, \Lambda/{\bf n} \otimes_{\Lambda} \psi)$
of the RHS. 
By 
the uniqueness of a pro-representable hull, 
it is enough to 
show that (\ref{isom:funct}) is formally smooth 
and bijective on $\mathbb{C}[t]/t^2$.
Note that since 
the functor (\ref{funct:M}) is 
represented by $M_{Q, I, \theta}(\gamma)$, 
the functor 
\begin{align*}
\mathrm{Def}_W \cneq 
\mathrm{Def}_{W}^{\n}|_{\aA rt^{\rm{loc}}}
\end{align*} 
is pro-represented by $\widehat{\oO}_{V, [W]}$. 
Therefore by Lemma~\ref{lem:prohull}, 
the natural transform 
(\ref{isom:funct}) is an isomorphism 
on $\aA rt^{\rm{loc}}$, 
and in particular bijective on $\mathbb{C}[t]/t^2$. 
It remains to
show that (\ref{isom:funct}) is formally smooth. 

Let $(f, \wW, \psi)$ be a 
$(\Lambda, {\bf n})$-point of 
the LHS of (\ref{isom:funct}), 
and suppose that 
$\wW$ extends to a flat
representation $\wW'$ of $(Q, I)$ over 
$\Lambda'$, where 
$\Lambda' \twoheadrightarrow \Lambda$ is a surjection in 
$\nN^{\rm{loc}}$. 
Then $\wW^{' ab}$ 
is a flat extension of $\wW^{ab}$
to $\Lambda^{'ab}$.  
By the pro-representability of $\mathrm{Def}_{W}$, 
the morphism $f \colon \Spec \Lambda^{ab} \to V$
 uniquely extends to 
a morphism of schemes 
$f' \colon \Spec \Lambda^{' ab} \to V$, 
such that there
 exists 
an isomorphism 
$\psi' \colon \wW^{' ab} \stackrel{\cong}{\to} f^{'\ast}\vV$. 
By Lemma~\ref{lem:I} (i), 
any automorphism of $f^{\ast}\vV$
extends to an automorphism of $f^{'\ast}\vV$, 
hence one can choose $\psi'$ so that 
$\Lambda^{ab} \otimes_{\Lambda^{'ab}}\psi'=\psi$
holds. 
Then the triple 
$(f', \wW', \psi')$ is an extension of 
$(f, \wW, \psi)$
to $\Lambda'$, 
showing that (\ref{isom:funct}) is formally smooth. 
\end{proof}
By Corollary~\ref{cor:NCV} together with
the above lemma, we obtain the following result:
\begin{thm}
There exists 
a quasi NC structure $\{(V_i, \oO_{V_i}^{\n})\}_{i\in \mathbb{I}}$
on $M_{Q, I, \theta}(\gamma)$
such that for any $[W] \in V_i$, 
there is an isomorphism of algebras 
$\widehat{\oO}_{V_i, [W]}^{\n} \cong R_W^{\n}$. 
\end{thm}

\subsection{Partial NC thickening of moduli spaces of representations of quivers}\label{subsec:partial}
It is not clear whether the quasi NC structure 
in Corollary~\ref{cor:NCV}
glue together to 
give an NC structure. 
Here we discuss 
the possibility to 
extend a given $(d-1)$-th order
NC thickening to 
that of the $d$-th order
NC thickening. 
Let $U_i^{\n}$, $V_i^{\n}$ be
affine NC schemes 
given in (\ref{UiVi}), 
and 
take representations 
$\vV_{U_i}^{\n}$, $\vV_{V_i}^{\n}$
of $(Q, I)$ as before.  
For $d \in \mathbb{Z}_{\ge 0}$, 
we set
\begin{align*}
U_i^{d} \cneq (U_i^{\n})^{\le d}, \ 
V_i^{d} \cneq (V_i^{\n})^{\le d}, \ 
\vV_{U_i}^{d} \cneq (\vV_{U_i}^{\n})^{\le d}, \ 
\vV_{V_i}^{d} \cneq (\vV_{V_i}^{\n})^{\le d}.
\end{align*}
The isomorphisms (\ref{Vij}), (\ref{isom:g})
in Corollary~\ref{cor:NCV}
induce
isomorphisms 
\begin{align*}
\phi_{ij}^{\le d} \colon 
V_j^{d}|_{V_{ij}} \stackrel{\cong}{\to} V_{i}^{d}|_{V_{ij}}, \
g_{ij}^{\le d} \colon 
\phi_{ij}^{\le d \ast}\vV_{V_j}^{d}|_{V_{ij}}
\stackrel{\cong}{\to}
\vV_{V_i}^{d}|_{V_{ij}}
\end{align*}
where $\phi_{ij}^{\le d}$
give
a quasi NC structure on $M_{Q, I, \theta}(\gamma)$, 
and 
$g_{ij}^{\le d}$ are isomorphisms of 
representations of $(Q, I)$. 
We put the following assumption: 
\begin{assum}
\label{assum:assum}
\begin{enumerate}
\item By replacing $\phi_{ij}^{\le d-1}$ if necessary, 
the quasi NC structure 
$\{V_i^{d-1}\}_{i \in \mathbb{I}}$ 
determines an NC structure
$M_{Q, I, \theta}^{d-1}(\gamma)$
on $M_{Q, I, \theta}(\gamma)$. 
\item By replacing 
$g_{ij}^{\le d-1}$
if necessary, 
the objects 
$\{\vV_{V_i}^{d-1}\}_{i\in \mathbb{I}}$
glue 
to give a representation $\vV^{d-1}$ of $(Q, I)$ over 
$M_{Q, I, \theta}^{d-1}(\gamma)$. 
\end{enumerate}
\end{assum}
We define
the functor
\begin{align*}
h_{Q, I, \theta}^{d}(\gamma)
\colon \nN_d \to \sS et
\end{align*}
by sending $\Lambda \in \nN_d$ to the 
isomorphism classes of triples 
$(f, \wW, \psi)$: 
\begin{itemize}
\item $f$ is a morphism of NC schemes 
$f\colon \Spf \Lambda^{\le d-1} \to M_{Q, I, \theta}^{d-1}(\gamma)$. 
\item $\wW$ is a flat representation of $(Q, I)$ over $\Lambda$. 
\item $\psi \colon \wW^{\le d-1} \stackrel{\cong}{\to}
f^{\ast}\vV^{d-1}$
is an isomorphism of 
representations of $(Q, I)$ over $\Lambda^{\le d-1}$. 
\end{itemize}
An isomorphism 
$(f, \wW, \psi) \to (f', \wW', \psi')$ 
exists if $f=f'$, and 
there is an isomorphism $\wW \to \wW'$ 
as representations of $(Q, I)$ over $\Lambda$ 
commuting with $\psi$, $\psi'$.
\begin{prop}\label{prop:naturald}
The natural transform
\begin{align}\label{nat3}
h_{V_i^{d}} \to h_{Q, I, \theta}^{d}(\gamma)|_{V_i}
\end{align}
sending 
$g \colon \Spf \Lambda \to V_i^{d}$
for $\Lambda \in \nN_d$ to 
$(g^{\le d-1}, g^{\ast}\vV_{V_i}^{d}, \id)$
is an isomorphism of functors. 
\end{prop}
\begin{proof}
Note that (\ref{nat3}) is an isomorphism 
on $\nN_{<d}$. 
Similarly to the proof of Proposition~\ref{prop:natural2}, 
we have the Cartesian square
\begin{align}\label{Car2}
\xymatrix{
h_{U_i^{d}} \ar[r] & h_{Q, \{0\}, \theta}^{d}(\gamma)|_{U_i} \\
h_{V_i^{d}} \ar[r] \ar@{^{(}->}[u]
 & h_{Q, I, \theta}^{d}(\gamma)|_{V_i}. \ar@{^{(}->}[u]
}
\end{align} 
It is enough to show that 
the top arrow of (\ref{Car2}) is an isomorphism. 
We write $h^{d}=h_{Q, \{0\}, \theta}^{d}(\gamma)|_{U_i}$
for simplicity. 
By Proposition~\ref{prop:hull2}, it 
remains to show the following: 
for surjections $p_1 \colon \Lambda_1 \twoheadrightarrow \Lambda$, 
$p_2 \colon \Lambda_2 \twoheadrightarrow \Lambda$
in $\nN_{d}$ with $\Lambda \in \nN_{d-1}$, the natural map
\begin{align}\label{map:hd}
h^d(\Lambda_{12}) \to h^d(\Lambda_1) \times_{h^d(\Lambda)} h^d(\Lambda_2)
\end{align}
is bijective. 
We follow
the same notation and argument
in the proof of
Proposition~\ref{prop:natural}, 
replacing $h$ with $h^d$, and 
$\ast^{ab}$ with $\ast^{\le d-1}$.  
The difference from the proof of Proposition~\ref{prop:natural}
is that, since 
we have
\begin{align*}
\Lambda \otimes_{\Lambda_j}\wW_j=
\left( \Lambda \otimes_{\Lambda_j} \wW_j \right)^{\le d-1}
\end{align*}
and
the isomorphism $\gamma$ in (\ref{isom:h}) 
should commute with $\Lambda \otimes_{\Lambda_j} \psi_j$, 
$\gamma$
is uniquely determined
by the RHS of (\ref{map:hd}). 
Therefore sending the triples (\ref{RHS})
to $(f_{12}, \wW_{12}, \psi_{12})$
is
a well-defined map from the RHS to 
the LHS of (\ref{map:hd}), 
giving the inverse of (\ref{map:hd}). 
\end{proof}
By the above proposition together with 
Remark~\ref{rmk:NC}, we have the following corollary: 
\begin{cor}\label{cor:glue}
Under Assumption~\ref{assum:assum},
affine NC structures $\{V_i^{d}\}_{i\in \mathbb{I}}$ glue together 
to give an NC structure 
$M_{Q, I, \theta}^{d}(\gamma)$ 
on $M_{Q, I, \theta}(\gamma)$.  
\end{cor}
By the above corollary, 
if the local $(Q, I)$ representations 
$\vV_{V_i}^{d}$
glue together to give a global $(Q, I)$
representation, then 
we can extend the $d$-th 
order NC structure $M_{Q, I, \theta}^d(\gamma)$
to a $(d+1)$-th order
NC thickening. 
The obstruction of
gluing $\vV_{V_i}^{d}$
is given as follows. We set
\begin{align*}
\jJ^d \cneq \Ker \left(\oO_{M_{Q, I, \theta}^{d}(\gamma)} \twoheadrightarrow 
\oO_{M_{Q, I, \theta}^{d-1}(\gamma)}\right).
\end{align*}
Note that $\jJ^d$ is a
coherent sheaf on $M_{Q, I, \theta}(\gamma)$. 
\begin{lem}\label{lem:obstr}
In the situation of Corollary~\ref{cor:glue},
there is a class
\begin{align}\label{glue:H2}
\mathrm{ob} \in 
H^2(M_{Q, I, \theta}(\gamma), \jJ^d)
\end{align}
such that $\mathrm{ob}=0$ if and only if 
$\{\vV_{V_i}^d\}_{i\in \mathbb{I}}$ glue to 
give a representation $\vV^d$ of $(Q, I)$
over $M_{Q, I, \theta}^{d}(\gamma)$.  
\end{lem}
\begin{proof}
By Lemma~\ref{lem:I} (ii), 
the automorphism
$g_{ij}^{\le d} \circ g_{jk}^{\le d} \circ g_{ki}^{\le d}$
of $\vV_{V_i}^{d}|_{V_{ijk}}$
 is given  
by the multiplication 
of 
$1+u_{ijk}$
for some element
$u_{ijk} \in \jJ_d|_{V_{ijk}}$. 
Then $\{u_{ijk}\}_{ijk}$
is 
a Cech 2-cocycle of (\ref{glue:H2}), 
giving the desired obstruction class. 
\end{proof}
Since Assumption~\ref{assum:assum} is 
always satisfied for $d=1$, 
the affine NC structures $\{V_i^{1}\}_{i\in \mathbb{I}}$ 
glue together 
to give an NC structure 
$M_{Q, I, \theta}^{1}(\gamma)$ 
on $M_{Q, I, \theta}(\gamma)$.
Indeed,
one of
such a thickening 
is explicitly constructed in the following way. 
We first construct a 1-smooth NC structure 
$M_{Q, \theta}^1(\gamma)$
on 
the smooth moduli space $M_{Q, \theta}(\gamma)$
by 
\begin{align*}
\oO_{M_{Q, \theta}^1(\gamma)}
=\oO_{M_{Q, \theta}(\gamma)}
\oplus \Omega^2_{M_{Q, \theta}(\gamma)}.
\end{align*}
The multiplication is given 
\begin{align*}
(x, f) \cdot (y, g)=(xy, xg+fy+dx \wedge dy). 
\end{align*}
The universal 
representation $\vV_{U_i}^1$
may not glue together, but
the ideal of relations $\jJ_{I, U_i}^1 \subset \oO_{U_i}^{1}$
coincide on $U_{ij}$, 
hence determines the two sided 
ideal $\jJ_{I}^{1} \subset \oO_{M_{Q, \theta}^1(\gamma)}$. 
Then the NC scheme
\begin{align*}
M_{Q, I, \theta}^1(\gamma)
=\left(M_{Q, I, \theta}(\gamma), 
\oO_{M_{Q, \theta}^1}(\gamma)/\jJ_I^1 \right)
\end{align*}
is 
isomorphic to $V_i^1$
on $V_i$.

\subsection{NC structures on framed quiver representations}
As we observed in Remark~\ref{rmk:aut}, 
the issue 
for 
having the global NC structure
is caused by some automorphisms of sheaves
over non-commutative bases. 
One of the classical ways 
to kill automorphisms
of sheaves
is 
to add additional data called \textit{framing}.
Here we show that 
this classical idea also works for 
the construction problem of global NC structures.  
Let $(Q, I)$ be a quiver
with relation.  
We fix a vertex $\star \in Q_0$. 
\begin{defi}
A framed representation of $(Q, I)$
 over an NC scheme 
$Y$ 
is a pair $(\wW, \tau)$, 
where $\wW$ is 
a representation of $(Q, I)$ over $Y$
and  
$\tau \colon \oO_Y \to \wW_{\star}$
is
a morphism of coherent left $\oO_Y$-modules.   
\end{defi}
A framed representation $(\wW, \tau)$ over $Y$
is called \textit{flat} if $\wW$ is flat. 
Let $(\wW', \tau')$ be 
another framed 
representation of $(Q, I)$ 
over $Y$. 
An isomorphism from $(\wW, \tau)$ 
to $(\wW', \tau')$
is an isomorphism 
$g \colon \wW \stackrel{\cong}{\to} \wW'$
in $\mathrm{Rep}((Q, I)/Y)$
such that $g_{\star} \circ \tau=\tau'$. 
\begin{rmk}
For $\Lambda \in \nN$, 
giving 
a framed representation of $(Q, I)$
over $\Spf \Lambda$ is equivalent 
to 
giving 
a representation $W$
of $(Q, I)$ over 
$\Lambda$
together with 
a left $\Lambda$-module 
homomorphism $\tau \colon \Lambda \to W_{\star}$.
The data $(W, \tau)$ is called a
framed representation of $(Q, I)$ over $\Lambda$.  
If $\Lambda=\mathbb{C}$, we 
just call it 
a framed representation of $(Q, I)$. 
\end{rmk}

\begin{defi}
For $\theta \in \Gamma_{Q}$, 
a framed 
representation $(W, \tau)$ of $(Q, I)$
is called $\theta$-stable 
if $\theta \cdot \dim W=0$, 
$W$ is $\theta$-semistable 
and for any
subobject $0\neq W' \subset W$
in $\mathrm{Rep}(Q)$
with $\Imm \tau \in W'_{\star}$, 
we have $\theta \cdot \dim W'>0$. 
\end{defi}
For $\theta, \gamma \in \Gamma_{Q}$ with $\theta \cdot \gamma=0$,
the functor 
\begin{align*}
\mM_{Q, I, \theta}^{\star}(\gamma) \colon 
\sS ch/\mathbb{C} \to \sS et
\end{align*}
is defined 
by sending a $\mathbb{C}$-scheme $T$
to the set of isomorphism 
classes of 
flat 
framed representations $(\wW, \tau)$
over $T$ such that 
for any $t\in T$, 
the restriction $(\wW|_{t}, \tau|_{t})$
is a $\theta$-stable framed 
representation 
of $(Q, I)$
with $\dim \wW|_{t}=\gamma$. 
\begin{prop}\label{prop:fN}
The functor $\mM_{Q, I, \theta}^{\star}(\gamma)$ is represented 
by a 
 quasi-projective scheme
$M_{Q, I, \theta}^{\star}(\gamma)$, which 
is non-singular if $I=\{0\}$. 
\end{prop}
\begin{proof}
Let $Q^{\diamond}$ be a quiver 
defined by adding one vertex $\diamond$
and one arrow $\diamond \to \star$
to the quiver $Q$. 
Note that the relation $I$ naturally 
determines the relation in $Q^{\diamond}$.
By the construction, giving
a framed representation of $(Q, I)$ is equivalent 
to giving a representation of 
$(Q^{\diamond}, I)$
whose dimension vector at the vertex $\diamond$ is one. 

For a
rational number $0<\varepsilon \ll 1$, 
let $\theta^{\diamond}_{\varepsilon}$ be an 
element of $\Gamma_{Q^{\diamond}}\otimes \mathbb{Q}$ 
given by 
\begin{align*}
\theta^{\diamond}_{\varepsilon, v}=\theta_{v}+\varepsilon
\mbox{ for } v\in Q_0, \  
\theta^{\diamond}_{\varepsilon, \diamond}=-\varepsilon
\sum_{v\in Q_0} \gamma_v.
\end{align*} 
Then $\theta_{\varepsilon}^{\diamond} \cdot \gamma^{\diamond}=0$, 
where
$\gamma^{\diamond}=(1, \gamma) \in \Gamma_{Q^{\diamond}}$
and $1$ is 
the dimension vector at the vertex $\diamond$.
Let $(\wW, \tau)$ be a flat 
framed representation of 
$(Q, I)$ over a $\mathbb{C}$-scheme $T$
such that $\dim \wW|_t=\gamma$ for any $t\in T$. 
Then we have the associated flat
representation 
$\wW^{\diamond}$
of $(Q^{\diamond}, I)$ over $T$
by putting $\wW_{\diamond}^{\diamond}=\oO_T$ 
and the morphism 
corresponding to $\diamond \to \star$
is $\tau$. 
It is easy to show that 
$(\wW, \tau)$ is a family of 
framed $\theta$-stable 
representations of $(Q, I)$
if and only if
$\wW^{\diamond}$ is 
a family of $\theta^{\diamond}$-stable
representations 
of $(Q^{\diamond}, I)$. 
Therefore $(\wW, \tau) \mapsto \wW^{\diamond}$
gives the natural transform
\begin{align}\label{nat}
\mM_{Q, I, \theta}^{\star}(\gamma) \to 
\mM_{Q^{\diamond}, I, \theta^{\diamond}_{\varepsilon}}(\gamma^{\diamond}). 
\end{align}
For a $T$-valued point $\wW^{\diamond}$
of the RHS of (\ref{nat}), 
it is equivalent to a unique element
$\wW^{'\diamond}$
such that $\wW^{'\diamond}_{\diamond}=\oO_T$, 
up to isomorphisms.  
Therefore the natural transform (\ref{nat}) is indeed
an isomorphism of functors, and the result follows from Theorem~\ref{thm:Kin}.   
\end{proof}
Let 
$(\vV, \iota)$ 
be the universal framed representation of $(Q, I)$ 
over
 $M_{Q, I, \theta}^{\star}(\gamma)$. 
We define the functor
\begin{align}\label{def:hQI}
h^{\star}_{Q, I, \theta}(\gamma) \colon \nN \to \sS et
\end{align}
by sending $\Lambda \in \nN$ to the 
isomorphism classes of triples 
$(f, (\wW, \tau), \psi)$: 
\begin{itemize}
\item $f$ is a morphism 
of schemes $\Spec \Lambda^{ab} \to M_{Q, I, \theta}^{\star}(\gamma)$. 
\item $(\wW, \tau)$ is a flat 
framed representation of $(Q, I)$
over $\Lambda$. 
\item $\psi$ is an isomorphism
$\psi \colon (\wW^{ab}, \tau^{ab}) \stackrel{\cong}{\to}
f^{\ast}(\vV, \iota)$ as framed representations of $(Q, I)$ 
over $\Lambda^{ab}$. 
\end{itemize}
We use the following lemma: 
\begin{lem}\label{lem:aut}
For $(f, (\wW, \tau), \psi) \in h_{Q, I, \theta}^{\star}(\gamma)(\Lambda)$, 
we have $\Aut(\wW, \tau)=\id$. 
\end{lem}
\begin{proof}
We prove the lemma
by the induction on the degree of the NC nilpotence of $\Lambda$.
First suppose that $\Lambda$ is commutative. 
We use the notation in the 
proof of Proposition~\ref{prop:fN}. 
Let $\wW^{\diamond}$ be the
$\Spec \Lambda$-valued point 
of  
$\mM_{Q^{\diamond}, I, \theta_{\varepsilon}^{\diamond}}(\gamma^{\diamond})$
corresponding to $(\wW, \tau)$ under 
the natural transform (\ref{nat}).  
Then by Lemma~\ref{lem:I} (i), 
any automorphism 
of $\wW^{\diamond}$ 
is given by a multiplication of 
an element in $\Lambda^{\ast}$. 
Therefore it commutes with $\tau$
if and only if it is identity, 
proving the lemma when $\Lambda$ is commutative.  
 
Next suppose that $\Lambda \in \nN_d$ 
and the lemma holds for $\Lambda^{\le d-1}$. 
Let $g$ be an automorphism of 
$(\wW, \tau)$. 
Then $g$ induces an automorphism of $\wW^{\diamond}$
which 
induces identity 
on
$(\wW^{\diamond})^{\le d-1}$.  
Hence by Lemma~\ref{lem:I} (ii), 
$g$ is a left multiplication by some central 
element in $\Lambda$. 
Since $g$ commutes with $\tau$, it follows that $g$ must be an identity. 
\end{proof}
Using the above lemma, 
we show the following result 
on the existence 
of global NC structures on the 
moduli spaces of stable framed representations. 
\begin{thm}\label{thm:framedQ}
The framed moduli space $M_{Q, I, \theta}^{\star}(\gamma)$
has an NC structure 
which represents 
the functor
$h^{\star}_{Q, I, \theta}(\gamma)$. 
\end{thm}
\begin{proof}
In order to 
simplify the notation, we write 
$M=M_{Q, I, \theta}^{\star}(\gamma)$
and 
$N=M_{Q, \{0\}, \theta}^{\star}(\gamma)$. 
Note that $N$ is non-singular. 
Let $U \subset N$ be a sufficiently small 
affine open subset, 
and write 
$h^{\star}=h^{\star}_{Q, \{0\}, \theta}(\gamma)|_{U}$. 
Let $U^{\n}$ be an NC smooth thickening of $U$. 
Similarly to 
Proposition~\ref{prop:natural},
we can construct a 
natural transform $h_{U^{\n}} \to h^{\star}$, 
which is an isomorphism on $\cC om$.
We show that $h_{U^{\n}} \to h^{\star}$ is 
indeed
an isomorphism on $\nN$. 
By Proposition~\ref{prop:hull2}, 
it is enough to show 
the following: 
for surjections $p_j \colon \Lambda_{j} \twoheadrightarrow \Lambda$
in $\nN$ with $j=1, 2$, 
$\Lambda_{12} \cneq \Lambda_1 \times_{\Lambda} \Lambda_2$, 
the natural map
\begin{align}\label{map:star}
h^{\star}(\Lambda_{12}) \to h^{\star}(\Lambda_1) \times_{h^{\star}(\Lambda)}
h^{\star}(\Lambda_2)
\end{align}
is a bijection. 
Let us take an element of the RHS of (\ref{map:star}), i.e. 
elements of 
$h^{\star}(\Lambda_1)$
and $h^{\star}(\Lambda_2)$
which 
are isomorphic over $\Lambda$.  
Following the proof of Proposition~\ref{prop:natural}, 
one can lift it to
an element of the LHS of (\ref{map:star}). 
By Lemma~\ref{lem:aut}, 
the framed isomorphism 
over $\Lambda$ is uniquely determined,  
hence
the similar argument of the proof of
Proposition~\ref{prop:naturald}
shows that  
the above 
lift is uniquely 
determined. 
Thus we obtain a map from the RHS
to the LHS of (\ref{map:star}), which obviously 
gives the inverse of (\ref{map:star}). 
Therefore $h_{U^{\n}} \to h^{\ast}$ 
is an isomorphism. 

By Remark~\ref{rmk:NC}, 
the affine NC structures
$U^{\n}$ 
on each affine open subset $U \subset N$
glue together to give 
the NC structure $(N, \oO_N^{\n})$
on $N$.
Again by Lemma~\ref{lem:aut}, 
the local framed universal representation on 
$U^{\n}$
also 
glue to give 
the global universal framed
representation $(\vV^{\n}, \tau^{\n})$
on $N^{\n}$, which 
induces the functorial 
isomorphism 
$h_{N^{\n}} \to h_{Q, \{0\}, \theta}(\gamma)$
by the above argument.  

Now we consider the subscheme $M \subset N$. 
Let $\jJ_{I} \subset \oO_{N^{\n}}$ be 
the ideal of relation in $I$, 
and set $\oO_M^{\n}=\oO_N^{\n}/\jJ_I$. 
Then 
as in the proof of Proposition~\ref{prop:natural2}, 
the NC scheme
$(M, \oO_M^{\n})$ is 
the desired NC structure 
by the 
Cartesian square: 
\begin{align*}
\xymatrix{
h_{N^{\n}} \ar[r] & h_{Q, \{0\}, \theta}(\gamma) \\
h_{M^{\n}} \ar[r] \ar@{^{(}->}[u] & h_{Q, I, \theta}(\gamma). \ar@{^{(}->}[u]
}
\end{align*}
\end{proof}

\begin{rmk}
By the proof of Theorem~\ref{thm:framedQ}, 
the framed smooth moduli space
$M_{Q, \{0\}, \theta}^{\star}(\gamma)$
has a NC smooth thickening. 
They give many examples of 
varieties which admit NC 
smooth thickening. 
\end{rmk}

\subsection{An example}\label{subsec:anex}
We describe an example of an 
NC thickening of 
the moduli space of representations of a quiver. 
We consider the quiver $Q$
described as 
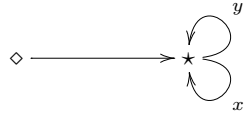
\begin{align*}
\xymatrix{
\diamond \ar[rr]
&& \star \ar@(r, d)^{x}
\ar@(r, u)_{y}
 }
\end{align*}
with  
relation $I$ given by
\begin{align}\label{rel:xy}
xy=yx. 
\end{align}
We also set the 
vectors $\gamma$ and $\theta$
in $\Gamma_{Q}$
to be
\begin{align*}
\gamma=
(\gamma_{\diamond}, \gamma_{\star})=(1, 2), \ 
\theta=
(\theta_{\diamond}, \theta_{\star})=(-2, 1). 
\end{align*}
A representation 
$W$ of $(Q, I)$ 
with dimension vector $\gamma$
is 
given by the diagram
\begin{align}\label{pic:W}
\xymatrix{
\mathbb{C} \ar[rr]^{f}
&& \mathbb{C}^2 \ar@(r, d)^{A}
\ar@(r, u)_{B}
 }
\end{align}
where
 $f, A, B$ are linear 
maps satisfying $AB=BA$. 
It is easy to see that a $(Q, I)$-representation (\ref{pic:W})
is
$\theta$-stable if and only if 
$f(1)$
generates $\mathbb{C}^2$
as $\mathbb{C}[x, y]$-module. Hence
we have 
the natural identification
\begin{align*}
M_{Q, I, \theta}(\gamma)=\Hilb^2(\mathbb{C}^2)
\end{align*}
where the RHS is the Hilbert scheme of two points on $\mathbb{C}^2$. 
On the other hand, the smooth 
moduli space $M_{Q, \theta}(\gamma)$
is given by
\begin{align*}
M_{Q, \theta}(\gamma)=
\left( \mathbb{C}^2 \times M_2(\mathbb{C}) \times M_2(\mathbb{C}) \right)^{s}
/\mathrm{GL}_{2}(\mathbb{C})
\end{align*}
Here $(-)^s$ means the $\theta$-stable part, and 
the
$\mathrm{GL}_2(\mathbb{C})$ 
action is given by
\begin{align*}
g(v, A, B)=(g^{-1}v, g^{-1}Ag, g^{-1}Bg). 
\end{align*}
By the $\theta$-stability, 
we have $v\neq 0$, hence we have
\begin{align}\label{MQs}
M_{Q, \theta}(\gamma)=
\left( \{(1, 0)^{t}\} \times M_2(\mathbb{C}) \times M_2(\mathbb{C}) \right)^{s}
/G
\end{align}
where $G$ is the stabilizer of $(1, 0)^{t}$
\begin{align*}
G=\left( \begin{array}{cc}
1 & u \\
0 & v
\end{array}  \right), \ 
u, v \in \mathbb{C}. 
\end{align*}
We omit $\{(1, 0)^t\}$ in the notation 
of the RHS of (\ref{MQs}). 
Then we have
\begin{align*}
\left(M_2(\mathbb{C}) \times M_2(\mathbb{C}) \right)^{s}
=\left\{ A=\left( \begin{array}{cc}
a_1 & a_2 \\
a_3 & a_4 
\end{array}
\right), \ 
B=\left( \begin{array}{cc}
b_1 & b_2 \\
b_3 & b_4 
\end{array}
\right) :
a_3 \neq 0 \mbox{ or } b_3 \neq 0
\right\}. 
\end{align*}
The open subsets 
$a_3 \neq 0$, $b_3 \neq 0$ are $G$-invariants, hence 
we obtain the open covering
\begin{align*}
M_{Q, \theta}(\gamma)=U_A \cup U_B
\end{align*}
where $U_A$, $U_B$ are quotients of $a_3 \neq 0$, $b_3 \neq 0$, 
respectively. 
We also obtain the open cover
\begin{align*}
M_{Q, I, \theta}(\gamma)=V_A \cup V_B, \ 
V_{\ast}=U_{\ast} \cap M_{Q, I, \theta}(\gamma). 
\end{align*}
For example, $U_A$ is given by
\begin{align*}
U_A=\left\{ A=\left( \begin{array}{cc}
0 & a_2 \\
1 & a_4 
\end{array}
\right), \ 
B=\left( \begin{array}{cc}
b_1 & b_2 \\
b_3 & b_4 
\end{array}
\right) :
a_i, b_i \in \mathbb{C}
\right\}
\end{align*}
so $U_A \cong \mathbb{C}^6$. 
A smooth NC thickening of $U_A$ is given by
\begin{align*}
U_A^{\n}=\mathrm{Spf}~\mathbb{C}\langle a_2, a_4, b_1, b_2, b_3, b_4 
\rangle_{[[ab]]}. 
\end{align*}
Then the ideal of 
relation (\ref{rel:xy})
in $\oO_{U_A}^{\n}$ is determined by the relation
\begin{align*}
\left( \begin{array}{cc}
0 & a_2 \\
1 & a_4 
\end{array}
\right)
\left( \begin{array}{cc}
b_1 & b_2 \\
b_3 & b_4 
\end{array}
\right)
=\left( \begin{array}{cc}
b_1 & b_2 \\
b_3 & b_4 
\end{array}
\right)
\left( \begin{array}{cc}
0 & a_2 \\
1 & a_4 
\end{array}
\right)
\end{align*}
where we regard $a_i$ and $b_i$ as non-commutative variables
in $\oO_{U_A}^{\n}$. 
By expanding the above matrix multiplications, 
we obtain 
\begin{align}\label{mat:ex}
\left( \begin{array}{cc}
a_2 b_3 -b_2 & a_2 b_4-b_1 a_2-b_2 a_4 \\
b_1+a_4 b_3-b_4 & b_2+a_4 b_4-b_3 a_2 -b_4 a_4 
\end{array}
\right)=0. 
\end{align}
By definition, 
two sided ideal 
$\jJ_A \subset \oO_{U_A}^{\n}$
of relations in $I$
is generated by the 
matrix components of the LHS of (\ref{mat:ex})
\begin{align*}
\jJ_A=( a_2 b_3-b_2, a_2 b_4-b_1 a_2-b_2 a_4, 
b_1 + a_4 b_3 -b_4, b_2 + a_4 b_4-b_3 a_2 -b_4 a_4 ). 
\end{align*}
The quotient 
$\oO_{V_A}^{\n}=\oO_{U_A}^{\n}/\jJ_A$
gives an
NC thickening of $V_A$
\begin{align*}
V_A^{\n}=\mathrm{Spf}~\frac{\mathbb{C}\langle a_2, a_4, b_1, b_3 
\rangle_{[[ab]]}}
{([a_2, b_1]+
a_2[a_4, b_3], [a_2, b_3]+[a_4, b_1]+a_4[a_4, b_3])}. 
\end{align*}
Note that $V_A \cong \mathbb{C}^4$, with
coordinates $(a_2, a_4, b_1, b_3)$. 
In the same way, 
we obtain an NC thickening of 
$V_B\cong \mathbb{C}^4$
\begin{align*}
V_B^{\n}=
\mathrm{Spf}~\frac{\mathbb{C}\langle b_2', b_4', a_1', a_3' 
\rangle_{[[ab]]}}
{([b_2', a_1']+
b_2'[b_4', a_3'], [b_2', a_3']+[b_4', a_1']+b_4'[b_4', a_3'])}.
\end{align*}
Note that
\begin{align*}
V_A \cap V_B =\{b_3 \neq 0\}=\{a_3' \neq 0\}. 
\end{align*}
The gluing isomorphism
\begin{align*}
V_A^{\n}|_{V_A \cap V_B} \stackrel{\cong}{\to}
V_B^{\n}|_{V_A \cap V_B}
\end{align*}
is calculated as
\begin{align*}
&a_3' \mapsto b_3^{-1} \\
&a_1' \mapsto -b_1 b_3^{-1} \\
&b_4' \mapsto b_1+b_3^{-1}b_1 b_3 + b_3^{-1}a_4 b_3^2 \\
&b_2' \mapsto a_2 b_3^2-b_1 b_3^{-1}b_1 b_3 -b_1 b_3^{-1} a_4 b_3^2. 
\end{align*}

\section{Quasi NC structures on moduli space of stable sheaves}\label{sec:NC3}
In this section, using the results in the previous 
section, we show the existence of 
quasi NC structures on the moduli spaces of 
stable sheaves on projective schemes
satisfying the desired property in Theorem~\ref{intro:main}. 
\subsection{Graded algebras}
Let $(X, \oO_X(1))$ be a polarized 
projective scheme over $\mathbb{C}$.
Then $X=\mathrm{Proj}(A)$ for
the graded $\mathbb{C}$-algebra $A$
given by
\begin{align*}
A=\bigoplus_{i\ge 0} H^0(X, \oO_X(i)). 
\end{align*}
Below, we set 
$\mathfrak{m}\cneq A_{>0}$ 
the maximal ideal of $A$. 

For $\Lambda \in \nN$, 
we set
$A_{\Lambda} \cneq A \otimes_{\mathbb{C}} \Lambda$. 
Let $A_{\Lambda} \modu_{\gr}$
be the category of finitely generated 
graded left $A_{\Lambda}$-modules. 
For $M \in A_{\Lambda} \modu_{\rm{gr}}$, we denote by $M_i$ the 
degree $i$-part of $M$, 
and $M(j)$ the graded left $A_{\Lambda}$-module 
such that $M(j)_{i}=M_{j+i}$. 
For 
an interval $\mathbb{I} \subset \mathbb{Z}$, 
we define  
\begin{align*}
A_{\Lambda}
\modu_{\mathbb{I}}  \subset A_{\Lambda} \modu_{\rm{gr}}
\end{align*}
to be 
the subcategory of graded 
left $A_{\Lambda}$-modules $M$ 
such that
$M_i=0$ for $i\notin I$. 
For $q>p>0$, we
describe the category 
$A_{\Lambda} \modu_{[p, q]}$
in terms of a quiver with relation.

We define the quiver $Q_{[p, q]}$ 
whose set of vertices is
\begin{align*}
\{p, p+1, \cdots, q\}. 
\end{align*}
The number of arrows 
of $Q_{[p, q]}$ from $i$ to $j$ is given by   
$\dim_{\mathbb{C}}\mathfrak{m}_{j-i}$
(cf.~Figure~\ref{fig:quiver}).  
Below we fix bases of $\mathfrak{m}_k$
for each $k\in \mathbb{Z}_{\ge 1}$, 
and identify the set of arrows from $i$ to $j$ with 
the set of 
basis elements of $\mathfrak{m}_{j-i}$. 
The multiplication
\begin{align}
\vartheta \colon 
\mathfrak{m}_{j-i} \otimes_{\mathbb{C}} \mathfrak{m}_{k-j}
 \to \mathfrak{m}_{k-i}
\end{align}
in $A$
defines the relation in $Q_{[p, q]}$, 
by defining the two sided 
ideal $I \subset \mathbb{C}[Q_{[p, q]}]$
to be generated by  
\begin{align*}
\vartheta(\alpha \otimes \beta)-\alpha \cdot \beta, \ 
\alpha \in \mathfrak{m}_{i-j}, \ \beta \in \mathfrak{m}_{k-j}.
\end{align*}
Here we have regarded $\alpha$, $\beta$ as 
formal linear combinations of paths from $i$ to
$j$, $j$ to $k$, respectively, 
and $\alpha \cdot \beta$ is the multiplication in 
$\mathbb{C}[Q_{[p, q]}]$. 
From the construction of $(Q_{[p, q]}, I)$, 
sending $(\{W_i\}_{i=p}^{q}, \{\phi_a\})$ 
to $\oplus_{i=p}^{q}W_i$ 
gives  
the equivalence
\begin{align}\label{equiv}
\mathrm{Rep}((Q_{[p, q]}, I)/\Lambda) \stackrel{\sim}{\to}
A_{\Lambda}\modu_{[p, q]} . 
\end{align}
\begin{figure}\label{fig:quiver}
\begin{align*}
\xymatrix{ \stackrel{p}{\bullet} 
\ar@<0ex>[r]^{\times 3}
\ar@/^/@<2ex>[rr]^{\times 6}
\ar@/_/@<-2ex>[rrr]_{\times 10}
\ar@/^/@<5ex>[rrrr]^{\times 15}
\ar@/_/@<-5ex>[rrrrr]_{\times 21}
& \stackrel{p+1}{\bullet}
\ar@<0ex>[r]^{\times 3}
\ar@/^/@<2ex>[rr]^{\times 6}
\ar@/_/@<-2ex>[rrr]_{\times 10}
\ar@/^/@<5ex>[rrrr]^{\times 15}
& \stackrel{p+2}{\bullet}
\ar@<0ex>[r]^{\times 3}
\ar@/^/@<2ex>[rr]^{\times 6}
\ar@/_/@<-2ex>[rrr]_{\times 10}
& \stackrel{p+3}{\bullet}
\ar@<0ex>[r]^{\times 3}
\ar@/^/@<2ex>[rr]^{\times 6}
& \stackrel{p+4}{\bullet}
\ar@<0ex>[r]^{\times 3}
& \stackrel{p+5}{\bullet}
 }
\end{align*}\caption{Quiver $Q_{[p, p+5]}$ 
for $X=\mathbb{P}^2$}
\end{figure}
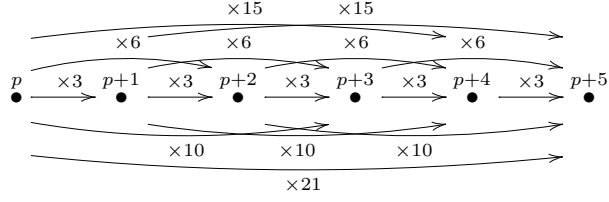
\subsection{Moduli stacks of semistable sheaves}
For $F \in \Coh(X)$,
let $\alpha(F, t)$ be its Hilbert polynomial
\begin{align*}
\alpha(F, t) \cneq \chi(F \otimes \oO_X(t))
\end{align*}
and $\overline{\alpha}(F, t)=\alpha(F, t)/c$ 
its reduced Hilbert polynomial, where 
$c$ is the leading coefficient of $\alpha(F, t)$.  
We recall the notion of (semi)stable sheaves (cf.~\cite{Hu}). 
\begin{defi}
A coherent sheaf $F$ on $X$
is called (semi)stable if 
it is a pure sheaf, 
and for any 
subsheaf $0 \subsetneq F' \subsetneq F$, we have 
\begin{align}\label{def:stab}
\overline{\alpha}(F', k) <(\le) \overline{\alpha}(F, k), \ 
k\gg 0.   
\end{align}
\end{defi}
Let us take a polynomial $\alpha \in \mathbb{Q}[t]$, 
which is a Hilbert polynomial of some coherent sheaf on $X$.  
The moduli stack
\begin{align}\label{stack:M}
\mathfrak{M}_{\alpha}^{} \colon 
\sS ch/\mathbb{C} \to \gG roupoid
\end{align}
is defined by sending 
a $\mathbb{C}$-scheme $T$
to the groupoid of 
$T$-flat 
sheaves $\fF \in \Coh(X \times T)$
such that for any $t\in T$, 
the sheaf $\fF_t \cneq \fF|_{X \times \{t\}}$
is $\theta$-semistable with Hilbert polynomial $\alpha$. 
The stack (\ref{stack:M}) is known to be
an algebraic stack of finite type over $\mathbb{C}$. 

Following~\cite{BFHR}, 
we relate the moduli stack 
$\mathfrak{M}_{\alpha}$ with
the moduli stack of semistable 
representations of $(Q_{[p, q]}, I)$ for $q\gg p \gg 0$. 
Let $T$ be a $\mathbb{C}$-scheme of finite type
and 
$\fF \in \Coh(X \times T)$
an object of
$\mathfrak{M}_{\alpha}(T)$.
Below we write 
$p_X \colon X \times T \to X$, 
$p_T \colon X \times T \to T$
the projections, 
and we 
set $\fF(i) \cneq \fF \otimes p_X^{\ast}\oO_X(i)$
for $i\in \mathbb{Z}$. 
For $q\gg p \gg 0$
and $i\in [p, q]$, 
we have $\dR p_{T\ast} \fF(i)=p_{T\ast} \fF(i)$, 
and the sheaf
\begin{align}\label{oplus}
\Gamma_{[p, q]}(\fF) \cneq 
\bigoplus_{i=p}^{q} p_{T\ast} \fF(i)
\end{align}
is a locally free
sheaf on $T$. 
In fact, the sheaf (\ref{oplus}) is 
a sheaf of graded $A \otimes_{\mathbb{C}} \oO_{T}$ algebras
whose graded pieces are concentrated on $[p, q]$. 
Hence 
by the equivalence (\ref{equiv}),  
$\Gamma_{[p, q]}(\fF)$ 
is a flat 
representation of $(Q_{[p, q]}, I)$
over $T$
with dimension vector 
\begin{align*}
\alpha_{[p, q]} \cneq 
(\alpha(p), \alpha(p+1), \cdots, \alpha(q)).
\end{align*}
Now we consider the stability
condition on $(Q_{[p, q]}, I)$. 
We set $\theta \in \Gamma_{Q_{[p, q]}}$ to be
\begin{align*}
\theta_p=-\alpha(q), \ \theta_q=\alpha(p), \ 
\theta_i=0 \mbox{ for } i\neq p, q. 
\end{align*}
By~\cite[Theorem~3.7]{BFHR},
the representation (\ref{oplus}) of $(Q_{[p, q]}, I)$
is a flat family of $\theta$-semistable representations over $T$. 
\begin{thm}\emph{(\cite[Corollary~3.4]{BFHR})}\label{thm:BFHR}
There exist $q \gg p \gg 0$
such that the 
functor $\fF \mapsto \Gamma_{[p, q]}(\fF)$
defines a 
morphism of algebraic stacks
\begin{align}\label{Gpq:stack} 
\mathfrak{M}_{\alpha}^{} \to 
\mathfrak{M}_{Q_{[p, q]}, I, \theta}^{}(\alpha_{[p, q]})
\end{align}
which is an open immersion. 
\end{thm}
Let $\mathfrak{M}_{[p, q]}$
be the image of the morphism (\ref{Gpq:stack}). 
By Theorem~\ref{thm:BFHR}, 
we have the isomorphism of stacks
\begin{align*}
\Gamma_{[p, q]} \colon 
\mathfrak{M}_{\alpha}^{}
\stackrel{\cong}{\to}
\mathfrak{M}_{[p, q]}. 
\end{align*}

\subsection{Non-commutative thickening of moduli stacks}
For $\Lambda \in \nN$, we 
denote by $X_{\Lambda}$
the NC scheme defined by $X \times \Spf \Lambda$. 
We define the 2-functor
\begin{align}\label{def:Mnc}
\mathfrak{M}_{\alpha}^{\n} \colon \nN \to \gG roupoid
\end{align}
by sending $\Lambda \in \nN$ to the 
groupoid of $\fF \in \Coh(X_{\Lambda})$ which 
is flat 
over $\Lambda$
such that $\fF^{ab} \in \mathfrak{M}_{\alpha}(\Spec \Lambda^{ab})$. 
Similarly we define the 2-functor
\begin{align}\label{def:Mnc2}
\mathfrak{M}_{[p, q]}^{\n}
\colon \nN \to \gG roupoid
\end{align}
by sending $\Lambda \in \nN$ 
to the groupoid of 
$\wW \in \mathrm{Rep}((Q_{[p, q]}, I)/\Lambda)$
flat over $\Lambda$
such that 
$\wW^{ab} \in \mathfrak{M}_{[p, q]}(\Spec \Lambda^{ab})$. 
Note that on the subcategory $\cC om \subset \nN$, 
 (\ref{def:Mnc}) and (\ref{def:Mnc2}) 
coincide with $\mathfrak{M}_{\alpha}$
and $\mathfrak{M}_{[p, q]}$ respectively. 
By Theorem~\ref{thm:BFHR},
the functor $\Gamma_{[p, q]}$ gives the isomorphism
for $q \gg p \gg 0$: 
\begin{align}\label{funct:isom}
\Gamma_{[p, q]} \colon 
\mathfrak{M}_{\alpha}^{\n}|_{\cC om} \stackrel{\cong}{\to}
\mathfrak{M}_{[p, q]}^{\n}|_{\cC om}. 
\end{align} 
Below we fix such $q \gg p \gg 0$. 
\begin{lem}\label{lem:projective}
For $\Lambda \in \nN$
and $\fF \in \mathfrak{M}_{\alpha}^{\n}(\Lambda)$, 
we have $\dR \Gamma(\fF(i))=\Gamma(\fF(i))$
for $i\in [p, q]$, and 
the
left $\Lambda$-module 
\begin{align}\label{GpqF}
\Gamma_{[p, q]}(\fF) =\bigoplus_{i\in [p, q]}
\Gamma(\fF(i))
\end{align}
is flat over $\Lambda$. 
Moreover the natural
morphism 
\begin{align}\label{natG}
u \colon \Gamma_{[p, q]}(\fF)^{ab} \to \Gamma_{[p, q]}(\fF^{ab})
\end{align}
is 
an isomorphism
of representations of $(Q_{[p, q]}, I)$
over $\Lambda^{ab}$. 
\end{lem}
\begin{proof}
Since $\fF^{ab} \in \mathfrak{M}_{\alpha}(\Spec \Lambda^{ab})$, 
we have 
$\dR \Gamma(\fF^{ab}(i))=
\Gamma(\fF^{ab}(i))$
for $i \in [p, q]$
by our choice of
$q \gg p \gg 0$. 
Note that $\fF$
admits a filtration whose subquotient
is given by 
$\gr_F^{\bullet}(\Lambda) \otimes_{\Lambda^{ab}} \fF^{ab}$. 
By the projection formula, 
for $i\in [p, q]$ we have 
\begin{align}\notag
\dR \Gamma(\gr_{F}^{\bullet}(\Lambda) \otimes_{\Lambda^{ab}} \fF^{ab}(i))
\cong \gr_{F}^{\bullet}(\Lambda) \otimes_{\Lambda^{ab}}\Gamma(\fF^{ab}(i)). 
\end{align}
Therefore we have
$\dR \Gamma(\fF(i)) = \Gamma(\fF(i))$
for $i\in [p, q]$. 
Also by the derived base change 
we have
\begin{align}\notag
\Lambda^{ab} \dotimes_{\Lambda} \dR \Gamma(\fF(i))
\cong \dR \Gamma(\fF^{ab}(i)). 
\end{align}
By combining
the above isomorphism with 
$\dR \Gamma(\fF^{ab}(i))=\Gamma(\fF^{ab}(i))$
and  
$\dR \Gamma(\fF(i))=\Gamma(\fF(i))$, 
we have the isomorphism
\begin{align*}
\Lambda^{ab} \dotimes_{\Lambda} 
\Gamma_{[p, q]}(\fF) \cong 
\Gamma_{[p, q]}(\fF^{ab}).
\end{align*}
Since 
$\Gamma_{[p, q]}(\fF^{ab})$ is flat over
$\Lambda^{ab}$, 
the above isomorphism 
implies that 
the left $\Lambda$-module 
(\ref{GpqF}) is 
flat
and
the morphism (\ref{natG}) is an isomorphism.
\end{proof}
By the above lemma, 
the isomorphism (\ref{funct:isom}) 
extends to the 1-morphism
\begin{align}\label{extend}
\Gamma_{[p, q]} \colon 
\mathfrak{M}_{\alpha}^{\n} \to 
\mathfrak{M}_{[p, q]}^{\n}
\end{align}
by sending 
$\fF$ to $\Gamma_{[p, q]}(\fF)$. 
\subsection{The left adjoint functor}
The purpose here is to construct 
the left adjoint of 
the functor (\ref{GpqF}). 
We first interpret $\Gamma_{[p, q]}$
in a derived categorical way. 
Let 
\begin{align*}
A_{\Lambda}
\modu_{\rm{tor}} \subset A_{\Lambda} \modu_{\gr}
\end{align*}
be the subcategory of finitely generated 
graded left $A_{\Lambda}$-modules $M$
with $M_i=0$ for $i\gg 0$. 
By Serre's theorem, we have
the equivalence
\begin{align}\label{Serre}
\Coh(X_{\Lambda}) \stackrel{\sim}{\to}
A_{\Lambda} \modu_{\gr}/A_{\Lambda} \modu_{\rm{tor}}
\end{align}
given by $\fF \mapsto \oplus_{i\ge 0} \Gamma(\fF(i))$. 
Below we identify the both sides of (\ref{Serre})
via the above equivalence.

On the other hand, we have the functor
\begin{align}\label{Ggep}
\dR \Gamma_{\ge p} \colon 
D^b (\Coh(X_{\Lambda})) \to D^b (A_{\Lambda} \modu_{\ge p})
\end{align}
defined by
\begin{align}\label{GpE}
\dR \Gamma_{\ge p}(E) \cneq 
\bigoplus_{i\ge p} \dR \Gamma(E(i)). 
\end{align}
The functor (\ref{Ggep})
has a left adjoint, given by the quotient functor
\begin{align}\label{pi:quot}
\pi \colon 
D^b (A_{\Lambda} \modu_{\ge p})
\to D^b \left(A_{\Lambda} \modu_{\gr}/A_{\Lambda} \modu_{\rm{tor}} \right).
\end{align}
Since $\pi \circ \dR \Gamma_{\ge p}=\id$, the 
functor (\ref{GpE}) is fully faithful. 

Let $\mathbb{C}(-i)$ be the one dimensional
graded $A$-module 
located in degree $i$. 
We define the subcategory 
\begin{align}\label{Spq}
\sS_{[p, q]} \subset D^b (A_{\Lambda} \modu_{\ge p})
\end{align}
to be the smallest triangulated
subcategory which contains 
objects of the form $\mathbb{C}(-i) \otimes_{\mathbb{C}} M$
for $i \in [p, q]$ and $M \in \Lambda \modu$. 
Similarly 
let
\begin{align*}
\pP_{[p, q]} \subset D^b(A_{\Lambda} \modu_{\ge p})
\end{align*} 
be the smallest triangulated subcategory 
which contains objects of the form 
$A(-i)
\otimes_{\mathbb{C}} M$
for $M\in \Lambda \modu$
and $i \in [p, q]$. 
\begin{lem}
We have the semiorthogonal decompositions:
\begin{align}\label{SOD1}
D^b(A_{\Lambda} \modu_{\ge p})&=\langle \sS_{[p, q]}, 
D^b (A_{\Lambda} \modu_{> q}) \rangle, \\
\label{SOD2}
D^b (A_{\Lambda} \modu_{\ge p})&=
\langle D^b(A_{\Lambda} \modu_{>q}), \pP_{[p, q]} \rangle.
\end{align}
\end{lem}
\begin{proof}
If $\Lambda=\mathbb{C}$, 
the result 
is proved in~\cite[Lemma~2.3]{Orsin}. 
Indeed the same argument works 
for any $\Lambda \in \nN$
to prove (\ref{SOD1}) and (\ref{SOD2}). 
\end{proof}
Note that the standard 
t-structure 
on $D^b (A_{\Lambda} \modu_{\ge p})$ restricts 
to the t-structure on $\sS_{[p, q]}$ whose 
heart is $A_{\Lambda}\modu_{[p, q]}$. 
Let $\dR \Gamma_{[p, q]}$ be the composition
\begin{align}\label{RGpq}
\dR \Gamma_{[p, q]} \colon 
D^b (\Coh(X_{\Lambda})) \stackrel{\dR \Gamma_{\ge p}}{\to}
D^b (A_{\Lambda} \modu_{\ge p}) \stackrel{pr_S}{\twoheadrightarrow} 
\sS_{[p, q]}. 
\end{align}
Here $pr_S$ is the projection with respect to the 
decomposition (\ref{SOD1}). 
By taking the zero-th cohomology, we obtain 
the functor
\begin{align}\label{Gpq1}
\Gamma_{[p, q]}\cneq 
\hH^0 \dR \Gamma_{[p, q]} \colon 
\Coh(X_{\Lambda}) \to A_{\Lambda} \modu_{[p, q]}
\end{align}
which coincides with (\ref{GpqF}).  

We describe the left adjoint of (\ref{RGpq}). 
Let $\dL \wp$ be the composition
\begin{align*}
\dL \wp \colon 
\sS_{[p, q]} \stackrel{i}{\hookrightarrow} 
&D^b(A_{\Lambda} \modu_{\ge p}) 
\stackrel{pr_P}{\twoheadrightarrow} \pP_{[p, q]} \\
&\stackrel{j}{\hookrightarrow} D^b(A_{\Lambda} \modu_{\ge p})
\stackrel{\pi}{\twoheadrightarrow} D^b (\Coh(X_{\Lambda})).  
\end{align*}
Here $i$, $j$ are the natural embeddings, 
and $pr_P$ is the projection with respect to  
the decomposition (\ref{SOD2}).
\begin{lem}\label{lem:adjoint}
$\dL \wp$ is the left adjoint of 
$\dR \Gamma_{[p, q]}$. 
\end{lem}
\begin{proof}
Since 
the functor $\pi$ in (\ref{pi:quot})
is the left adjoint of $\dR \Gamma_{\ge p}$ in 
(\ref{Ggep}),  
it is enough to show that the composition
\begin{align*}
\sS_{[p, q]} \stackrel{i}{\hookrightarrow} 
D^b(A_{\Lambda} \modu_{\ge p}) 
\stackrel{pr_P}{\twoheadrightarrow} \pP_{[p, q]} 
\stackrel{j}{\hookrightarrow} D^b(A_{\Lambda} \modu_{\ge p})
\end{align*}
 is the left 
adjoint of 
$pr_S \colon D^b (A_{\Lambda} \modu_{\ge p}) 
\twoheadrightarrow \sS_{[p, q]}$. 
We take 
$E \in \sS_{[p, q]}$ and 
$F \in D^b(A_{\Lambda} \modu_{\ge p})$. 
By the
decomposition (\ref{SOD2}), 
we have the distinguished triangle
\begin{align*}
j \circ pr_P \circ i(E) \to i(E) \to E'
\end{align*}
for some $E' \in D^b(A_{\Lambda} \modu_{>q})$. 
Applying $\Hom(-, i \circ pr_S(F))$ to the 
above triangle, the decomposition (\ref{SOD1}) shows 
that
\begin{align}\label{ad1}
\Hom(E, pr_S(F)) \stackrel{\cong}{\to}
\Hom(j\circ pr_P \circ i(E), i \circ pr_S(F)). 
\end{align}
By the decomposition (\ref{SOD1}), we also have the distinguished triangle
\begin{align*}
F' \to F \to i \circ pr_S(F)
\end{align*}
for some $F' \in D^b(A_{\Lambda} \modu_{>q})$. 
Applying $\Hom(j \circ pr_P \circ i(E), -)$
to the above triangle, we have the 
isomorphism
\begin{align}\label{ad2}
\Hom(j \circ pr_P \circ i(E), F) \stackrel{\cong}{\to}
\Hom(j\circ pr_P \circ i(E), i \circ pr_S(F)). 
\end{align}
The isomorphisms (\ref{ad1}), (\ref{ad2})
show that $j \circ pr_p \circ i$ is the 
left adjoint of $pr_S$. 
\end{proof}
\begin{lem}\label{lem:t-ex}
The functor $\dL \wp$ is 
right t-exact, i.e. 
\begin{align*}
\dL \wp \left(A_{\Lambda} \modu_{[p, q]} \right)
\subset D^{\le 0} (\Coh(X_{\Lambda})).
\end{align*}
\end{lem}
\begin{proof}
By the construction, 
the functor $\dR \Gamma_{[p, q]}$, 
is left t-exact, i.e. 
it takes $\Coh(X_{\Lambda})$
to 
$\sS^{\ge 0}_{[p, q]}$. 
Hence $\dL \wp$ is right t-exact
by Lemma~\ref{lem:adjoint}. 
\end{proof}
We define the following functor: 
\begin{align}\label{def:wp}
\wp \cneq \hH^0 \dL \wp
\colon A_{\Lambda} \modu_{[p, q]} \to
\Coh(X_{\Lambda}).
\end{align}
By Lemma~\ref{lem:adjoint}
and Lemma~\ref{lem:t-ex}, 
the above functor $\wp$
 is 
right exact, and gives the left 
adjoint functor of $\Gamma_{[p, q]}$ in (\ref{Gpq1}). 
\begin{prop}\label{inv:com}
The functor $\wp$ induces the 
1-morphism
\begin{align}\label{wp0}
\wp \colon 
\mathfrak{M}_{[p, q]}^{\n}|_{\cC om} \to \mathfrak{M}_{\alpha}^{\n}|_{\cC om}
\end{align}  
giving the inverse of (\ref{funct:isom}).
\end{prop}
\begin{proof}
Suppose that $\Lambda \in \nN$ is commutative. 
Since
the functor $\wp$ in (\ref{def:wp})
is the left adjoint functor of
$\Gamma_{[p, q]}$
in (\ref{Gpq1}), 
it 
coincides 
with the
left adjoint of 
$\Gamma_{[p, q]}$
constructed in~\cite[Proposition~3.1]{BFHR}.
By~\cite[Proposition~3.2]{BFHR},
the adjunction $\wp \circ \Gamma_{[p, q]} \to \id$
is an isomorphism on $\mathfrak{M}_{\alpha}$, 
hence $\wp$ gives the inverse of (\ref{funct:isom}).   
\end{proof}

\subsection{The inverse transform}
The purpose here 
is to show that $\wp$ in (\ref{wp0}) 
extends to
give the inverse of (\ref{extend}).
We prepare some lemmas. 
 \begin{lem}\label{lem:prepare}
For $\Lambda \in \nN$, let 
$M$ 
be a $\Lambda$ bi-module. 
Then 
the functors $\dR \Gamma_{[p, q]}$ and 
$\dL \wp$ commute with 
$M \dotimes_{\Lambda}-$. 
\end{lem}
\begin{proof}
The commutativity of $\dR \Gamma_{[p, q]}$ and 
$M \dotimes_{\Lambda}-$ 
follows from the derived base change 
$M \dotimes_{\Lambda} \dR \Gamma(\fF) \cong 
\dR \Gamma(M \dotimes_{\Lambda} \fF)$
for $\fF \in D^b(\Coh(X_{\Lambda}))$. 
The commutativity of 
$\dL \wp$ and $M \dotimes_{\Lambda}-$
follows from the construction of $\dL \wp$
and the fact that 
$M\dotimes_{\Lambda}-$ preserves the 
decomposition (\ref{SOD2}). 
\end{proof} 
\begin{lem}\label{lem:prepare2}
In the situation of Lemma~\ref{lem:prepare}, we have the following: 

(i) For $\fF \in \mathfrak{M}_{\alpha}^{\n}(\Lambda)$, 
we have $\Gamma_{[p, q]}(M \otimes_{\Lambda} \fF)
\cong M \otimes_{\Lambda} \Gamma_{[p, q]}(\fF)$.  

(ii) The functor $\wp$ in (\ref{def:wp}) commutes with $M \otimes_{\Lambda}-$. 
\end{lem}
\begin{proof}
(i) follows from Lemma~\ref{lem:projective} and Lemma~\ref{lem:prepare}, 
and (ii) follows from 
Lemma~\ref{lem:t-ex} and Lemma~\ref{lem:prepare}.
\end{proof}
\begin{lem}\label{lem:vanish}
By replacing $q \gg p \gg 0$ if necessary, 
for any $\Lambda \in \nN$ and 
$\fF \in \mathfrak{M}_{\alpha}^{\n}(\Lambda)$, 
the adjunction morphism
\begin{align*}
\wp \circ \Gamma_{[p, q]}(\fF) \to \fF
\end{align*}
is an isomorphism. 
Moreover we have  
$\hH^{-1}\left(\dL \wp \circ \Gamma_{[p, q]}(\fF)\right)=0$. 
\end{lem}
\begin{proof}
We take the cone of the adjunction morphism 
in $D^b(\Coh(X_{\Lambda}))$ 
\begin{align}\label{distt1}
\gG \to \dL \wp \circ \dR \Gamma_{[p, q]}(\fF) \to 
\fF. 
\end{align}
For a closed point 
$z=(x, y) \in X \times \Spec \Lambda^{ab}$, 
we regard its structure sheaf $\oO_z$
as an object of $\Coh(X_{\Lambda})$.
Note that  we have
\begin{align*}
\dR \Gamma_{[p, q]}(\fF)=\Gamma_{[p, q]}(\fF), \ 
\dR \Gamma_{[p, q]}(\oO_z)=\Gamma_{[p, q]}(\oO_z)
\end{align*}
by the proof of Lemma~\ref{lem:projective}. 
Hence applying $\dR \Hom_{X_{\Lambda}}(-, \oO_z)$ to the 
triangle (\ref{distt1}), 
we obtain the distinguished triangle 
\begin{align}\notag
\dR \Hom_{X_{\Lambda}}(\fF, &\oO_z) \to \\
&\dR \Hom_{A_{\Lambda}\rm{gr}}(\Gamma_{[p, q]}(\fF), 
\Gamma_{[p, q]}(\oO_z)) 
\label{distt2}
\to \dR \Hom_{X_{\Lambda}}(\gG, \oO_z). 
\end{align}
Since $\fF$ is flat over $\Lambda$, we have
\begin{align*}
\dR \Hom_{X_{\Lambda}}(\fF, \oO_z)=\dR \Hom_{X}(\fF^{ab}_y, \oO_x). 
\end{align*}
Here we 
have set $\fF^{ab}_y \cneq \fF^{ab}|_{X \times \{y\}}$. 
Also using Lemma~\ref{lem:prepare}, we have 
\begin{align*}
\dR \Hom_{A_{\Lambda} \gr}(\Gamma_{[p, q]}(\fF), 
\Gamma_{[p, q]}(\oO_z))
=\dR \Hom_{A \gr}(\Gamma_{[p, q]}(\fF^{ab}_{y}), 
\Gamma_{[p, q]}(\oO_x)). 
\end{align*}
As $\fF_y^{ab}$ corresponds to a closed point of $M_{\alpha}$, 
for $q \gg p \gg 0$ which are independent of 
$\fF$ and $(x, y)$, we have the isomorphisms
\begin{align*}
\Ext_X^i(\fF^{ab}_y, \oO_x) \stackrel{\cong}{\to}
\Ext_{A \gr}^i(\Gamma_{[p, q]}(\fF^{ab}_{y}), 
\Gamma_{[p, q]}(\oO_x))
\end{align*}
for $i\le 2$ by~\cite[Proposition~4.3.4]{FoKa}. 
Applying the above isomorphisms
to the triangle (\ref{distt2}), we obtain
\begin{align*}
\Hom_{X_{\Lambda}}(\gG, \oO_z[i])=0, \ 
i\le 1, \ z \in X \times \Spec \Lambda^{ab}
\end{align*}
which implies that 
$\hH^i(\gG)=0$ for $i\ge -1$. 
By
taking the long exact sequence of 
cohomologies 
associated to (\ref{distt1}), we 
obtain the desired result.  
\end{proof}

Now we show the following proposition: 
\begin{prop}\label{prop:isom}
The functor
$\wp$ induces the 1-morphism
\begin{align}\label{wpp}
\wp \colon \mathfrak{M}_{[p, q]}^{\n} \to 
\mathfrak{M}_{\alpha}^{\n}
\end{align}
giving the inverse of (\ref{extend}).
In particular, (\ref{extend}) is an isomorphism. 
\end{prop}
\begin{proof}
It is enough to show that, for 
$\Lambda \in \nN$ and $\wW \in \mathfrak{M}_{[p, q]}^{\n}(\Lambda)$, 
the object $\wp(\wW) \in \Coh(X_{\Lambda})$ is flat over 
$\Lambda$
and the adjunction morphism
\begin{align}\label{adjun}
\wW \to \Gamma_{[p, q]} \circ \wp(\wW)
\end{align}
is an isomorphism. 
Indeed
by Proposition~\ref{inv:com}
and
Lemma~\ref{lem:prepare2} (ii), 
we have 
\begin{align*}
\wp(\wW)^{ab} \cong \wp(\wW^{ab}) \in 
\mathfrak{M}_{\alpha}(\Spec \Lambda^{ab}). 
\end{align*}
Therefore 
if $\wp(\wW)$ is flat over $\Lambda$, then 
the 
object $\wp(\wW)$ determines 
an object of $\mathfrak{M}_{\alpha}^{\n}(\Lambda)$, 
and the 1-morphism 
(\ref{wpp}) is well-defined. 
Moreover if the morphism (\ref{adjun}) 
is an isomorphism, 
then combined with Lemma~\ref{lem:vanish}, 
the 1-morphism 
(\ref{wpp}) gives the inverse of 
(\ref{extend}). 
Below we prove the 
flatness of $\wp(\wW)$ and the isomorphism 
(\ref{adjun}) 
 by the induction of the NC nilpotence 
of $\Lambda$. 
The first step of the induction 
is the case of $\Lambda \in \cC om$, 
which follows from Proposition~\ref{inv:com}. 

Suppose that $\Lambda \in \nN_d$. 
By the assumption of the induction, 
we may assume that 
$\wp(\wW^{\le d-1})$ is flat over $\Lambda^{\le d-1}$
and 
the morphism (\ref{adjun})
is an isomorphism for $\wW^{\le d-1}$.
By Lemma~\ref{lem:prepare}, 
we have
the isomorphism
\begin{align}\notag
\Lambda^{\le d-1} \dotimes_{\Lambda} \dL \wp(\wW)
& \cong \dL \wp(\wW^{\le d-1})
\end{align}
which yields the spectral sequence
\begin{align}\label{spect}
E_{2}^{p, q}=
\tT or_{-p}^{\Lambda}(\Lambda^{\le d-1}, \hH^{q}(\dL \wp(\wW)))
\Rightarrow \hH^{p+q}(\dL \wp(\wW^{\le d-1})). 
\end{align}
On the other hand, 
we have the isomorphism
\begin{align*}
\dL \wp(\wW^{\le d-1}) \cong \dL \wp \circ \Gamma_{[p, q]} \circ 
\wp(\wW^{\le d-1})
\end{align*}
by the assumption of the induction. 
Applying Lemma~\ref{lem:vanish} to 
the above isomorphism, 
we obtain the vanishing
\begin{align}\label{Pvanish}
\hH^{-1}(\dL \wp(\wW^{\le d-1}))=0. 
\end{align}
The spectral sequence (\ref{spect})
together with the vanishing (\ref{Pvanish})
show that 
\begin{align}\label{torvanish}
\tT or_1^{\Lambda}(\Lambda^{\le d-1}, \wp(\wW))=0.
\end{align}
By Lemma~\ref{lem:prepare2} (ii), 
we have  
$\wp(\wW)^{\le d-1} \cong \wp(\wW^{\le d-1})$, 
which is flat over $\Lambda^{\le d-1}$
by the induction assumption. 
Therefore the vanishing (\ref{torvanish})
shows that $\wp(\wW)$ is flat over $\Lambda$. 

It remains to show the isomorphism (\ref{adjun}). 
 Let $J$ be the kernel of 
$\Lambda \twoheadrightarrow \Lambda^{\le d-1}$, which 
is $\Lambda^{ab}$-module. 
We have the exact sequence
\begin{align*}
0 \to J \otimes_{\Lambda^{ab}}
\wW^{ab} \to \wW \to \wW^{\le d-1} \to 0. 
\end{align*}
We apply $\wp$ to the above sequence. 
Since $\wp$ is right exact, 
using the vanishing 
(\ref{Pvanish}) and Lemma~\ref{lem:prepare2} (ii), 
we obtain the exact sequence 
\begin{align*}
0 \to J \otimes_{\Lambda^{ab}}
\wp(\wW^{ab}) \to \wp(\wW) \to \wp(\wW^{\le d-1}) \to 0. 
\end{align*}
Then we apply $\Gamma_{[p, q]}$
to the above sequence. 
By Lemma~\ref{lem:projective}
and Lemma~\ref{lem:prepare2} (i), 
we also have the exact sequence
\begin{align*}
0 \to J \otimes_{\Lambda^{ab}}
\Gamma_{[p, q]}\circ 
\wp(\wW^{ab}) \to 
\Gamma_{[p, q]} \circ \wp(\wW) \to 
\Gamma_{[p, q]} \circ \wp(\wW^{\le d-1}) \to 0. 
\end{align*}
We have the commutative diagram
of exact sequences
\begin{align*}
\xymatrix{
J \otimes_{\Lambda^{ab}}
\wW^{ab} \ar[r]\ar[d]^{\cong} & \wW
 \ar[r]\ar[d] & \wW^{\le d-1} \ar[d]^{\cong}\\
J \otimes_{\Lambda^{ab}} 
\Gamma_{[p, q]}\circ 
\wp(\wW^{ab}) \ar[r] & \Gamma_{[p, q]} \circ 
\wp(\wW) \ar[r] & \Gamma_{[p, q]} \circ 
\wp(\wW^{\le d-1}). 
}
\end{align*}
Here the right and left vertical arrows are isomorphisms 
by the assumption of the induction. 
By the five lemma, the morphism (\ref{adjun}) 
is an isomorphism. 
\end{proof}

\subsection{Quasi NC structures on the moduli space of stable sheaves}
Let 
\begin{align}\label{mfunct}
\mM_{\alpha} \colon 
\sS ch/\mathbb{C} \to \sS et
\end{align}
be the functor 
defined 
by sending
a $\mathbb{C}$-scheme $T$ to the
equivalence 
classes
of objects $\fF \in \mathfrak{M}_{\alpha}(T)$, 
where 
$\fF$ and $\fF'$
are called \textit{equivalent} if there is an line bundle $\lL$ on $T$
such that $\fF \cong \fF' \otimes p_T^{\ast}\lL$. 
The moduli functor (\ref{mfunct}) is not always
representable by a scheme, but 
if we assume that 
\begin{align}\label{primitive}
\mathrm{g. c. d}\{\alpha(m) : m\in \mathbb{Z}\}=1
\end{align}
then (\ref{mfunct})
is represented by a projective scheme $M_{\alpha}$
(cf.~\cite{Mu2}).  
Below we call $\alpha$ satisfying the condition 
(\ref{primitive}) as \textit{primitive}. 
In this case, 
the 
stack $\mathfrak{M}_{\alpha}$
consists of stable sheaves, and is a trivial
$\mathbb{C}^{\ast}$-gerbe
over $M_{\alpha}$. 

Suppose that $\alpha$ is primitive, and 
take $q \gg p \gg 0$ as in Theorem~\ref{thm:BFHR}. 
Let
\begin{align*}
\eE \in \Coh(X \times M_{\alpha})
\end{align*}
be a universal sheaf. 
Applying $\Gamma_{[p, q]}$ to $\eE$, 
we obtain a family of $\theta$-stable 
representations of $(Q_{[p, q]}, I)$ over 
$M_{\alpha}$. 
Note that if $\alpha$ is primitive, then $\alpha_{[p, q]}$ is 
a primitive dimension vector for $q \gg p \gg 0$.
Let $M_{Q_{[p, q]}, I, \theta}(\alpha_{[p, q]})$ be the 
moduli space of representations of $(Q_{[p, q]}, I)$ 
with dimension vector $\alpha_{[p, q]}$, 
given in Theorem~\ref{thm:Kin}. 
By Theorem~\ref{thm:BFHR}, 
the functor $\Gamma_{[p, q]}$ induces the 
morphism
\begin{align}\label{Gpq}
\Upsilon \colon M_{\alpha} \to M_{Q_{[p, q]}, I, \theta}(\alpha_{[p, q]})
\end{align}
which is an open immersion. 
We denote by 
\begin{align*}
M_{[p, q]} \subset M_{Q_{[p, q]}, I, \theta}(\alpha_{[p, q]})
\end{align*}
the image of the morphism (\ref{Gpq}). 
The scheme $M_{[p, q]}$ is an open 
subscheme of $M_{Q_{[p, q]}, I, \theta}(\alpha_{[p, q]})$, 
such that we have the isomorphism
\begin{align*}
\Upsilon \colon M_{\alpha} \stackrel{\cong}{\to}
M_{[p, q]}. 
\end{align*}
\begin{rmk}
Since $Q_{p, q}$
does not contain a loop, 
the 
moduli scheme 
$M_{Q_{[p, q]}, I, \theta}(\alpha_{[p, q]})$
is projective. 
Hence 
$M_{[p, q]}$ consists of 
union of connected components of 
$M_{Q_{[p, q]}, I, \theta}(\alpha_{[p, q]})$
\emph{(cf.~\cite[Corollary~3.8]{BFHR})}. 
\end{rmk}
We define the functor 
\begin{align*}
h_{\alpha} \colon \nN \to \sS et
\end{align*}
by sending $\Lambda \in \nN$ to the 
set of 
isomorphism classes of triples
$(f, \fF, \psi)$: 
\begin{itemize}
\item $f$ is a morphism of schemes 
$f \colon \Spec \Lambda^{ab} \to M_{\alpha}$. 
\item $\fF$ is an object of $\Coh(X_{\Lambda})$
which is flat over $\Lambda$. 
\item $\psi$ is an isomorphism $\psi \colon \fF^{ab} 
\stackrel{\cong}{\to} f^{\ast}\eE$. 
\end{itemize}
An isomorphism $(f, \fF, \psi) \to (f', \fF', \psi')$
exists if $f=f'$, 
and there is an isomorphism $\fF \to \fF'$ in $\Coh(X_{\Lambda})$
commuting $\psi$, $\psi'$. 
We also define the following functor: 
\begin{align*}
h_{[p, q]} \cneq h_{Q_{[p, q]}, I, \theta}(\alpha_{[p, q]})|_{M_{[p, q]}} \colon \nN \to \sS et. 
\end{align*}
Here 
$h_{Q_{[p, q]}, I, \theta}(\alpha_{[p, q]})$
is introduced in (\ref{h:Nset}). 
\begin{prop}\label{prop:isom2}
The functor $\Gamma_{[p, q]}$ in (\ref{Gpq1}) induces 
the isomorphism of functors
\begin{align*}
\Gamma_{[p, q]} \colon 
h_{\alpha}\stackrel{\cong}{\to}
h_{[p, q]}. 
\end{align*}
\end{prop}
\begin{proof}
The result obviously follows from 
the isomorphism $\Upsilon \colon M_{\alpha} \stackrel{\cong}{\to} M_{[p, q]}$
and Proposition~\ref{prop:isom}. 
\end{proof}
We have the following corollary: 
\begin{cor}\label{cor:NC}
The moduli scheme 
$M_{\alpha}$ 
admits an affine open cover
$\{U_i\}_{i\in \mathbb{I}}$, 
affine NC structures  
$\{U_i^{\n}=(U_i, \oO_{U_i}^{\n})\}_{i\in \mathbb{I}}$
and NC hulls 
$h_{U_i^{\n}} \to h_{\alpha}|_{U_i}$. In particular, 
there exist isomorphisms
 \begin{align}\label{phiU}
\phi_{ij} \colon U_j^{\n}|_{U_{ij}} 
\stackrel{\cong}{\to} U_i^{\n}|_{U_{ij}}
\end{align} 
of NC schemes giving a quasi NC structure on $M_{\alpha}$. 
\end{cor}
\begin{proof}
The result follows from 
Proposition~\ref{prop:natural2}, 
Corollary~\ref{cor:NCV}
and 
Proposition~\ref{prop:isom2}.
\end{proof}
Let $\eE_i^{\n}$
be the object of $\Coh(X \times U_i^{\n})$
corresponding to $\id \in h_{U_i^{\n}}(U_i^{\n})$
under 
the natural transformation 
$h_{U_i^{\n}} \to h_{\alpha}|_{U_i}$. 
Similarly to Corollary~\ref{cor:NCV}, 
we also have the isomorphisms
\begin{align}\label{isom:Ei}
g_{ij} \colon \phi_{ij}^{\ast}\eE_i^{\n}|_{U_{ij}}
 \stackrel{\cong}{\to}
\eE_j^{\n}|_{U_{ij}}. 
\end{align}

\subsection{Comparison with the formal deformations of sheaves}
Similarly to Subsection~\ref{subsec:compare}, 
we relate the 
quasi NC structure 
in Corollary~\ref{cor:NC}
with formal
non-commutative deformation
algebras of sheaves. 
For $F \in \Coh(X)$, the formal 
non-commutative deformation functor
\begin{align}\label{DefF}
\mathrm{Def}_{F}^{\n} \colon 
\nN^{\rm{loc}} \to \sS et
\end{align}
is defined by sending $(\Lambda, {\bf n})$
to the set of isomorphism classes 
$(\fF, \psi)$, where $\fF \in \Coh(X_{\Lambda})$
is flat over $\Lambda$
and 
$\psi \colon \Lambda/{\bf n} \otimes_{\Lambda} \fF \stackrel{\cong}{\to}F$ 
is an isomorphism in $\Coh(X)$. 
It is well-known that the 
formal commutative 
deformation space of $F$ 
is given by the solution of the Mauer-Cartan equation
of the differential graded algebra
$\dR \Hom(F, F)$, up to gauge equivalence. 
Let
\begin{align}\label{minimal}
(\Ext^{\ast}(F, F), \{m_n\}_{n\ge 2})
\end{align}
be the minimal $A_{\infty}$-algebra 
which is quasi-isomorphic to 
$\dR \Hom(F, F)$. 
The argument similar to~\cite{ESe}
shows that 
the pro-representable
hull of (\ref{DefF})
is described in terms of 
the $A_{\infty}$-structure of (\ref{minimal}). 
Let
\begin{align*}
m_n \colon \Ext^1(F, F)^{\otimes n} \to \Ext^2(F, F)
\end{align*}
be the $n$-th $A_{\infty}$-product, 
and 
\begin{align*}
J_F
\subset \widehat{T}^{\bullet}(\Ext^1(F, F)^{\vee})
\end{align*}
the topological closure of the two sided ideal 
generated by the image of the map
\begin{align*}
\sum_{n\ge 2} m_n^{\vee} 
\colon \Ext^2(F, F)^{\vee} \to \widehat{T}^{\bullet}(\Ext^1(F, F)^{\vee}). 
\end{align*}
The pro-representable hull of (\ref{DefF}) is given 
by the quotient algebra
\begin{align*}
R_{F}^{\n} \cneq 
\widehat{T}^{\bullet}(\Ext^1(F, F)^{\vee})/J_F.  
\end{align*}
The following 
is the main result in this section: 
\begin{thm}\label{main:thm}
There exists a 
quasi NC structure 
$\{U_i^{\n}=(U_i, \oO_{U_i}^{\n})\}_{i\in \mathbb{I}}$
on $M_{\alpha}$ 
such that for any $[F] \in U_i$, 
there is an isomorphism of algebras
$\widehat{\oO}_{U_i, [F]}^{\n} \cong R_{F}^{\n}$. 
\end{thm}
\begin{proof}
We take the quasi NC structure on $M_{\alpha}$ as in 
Corollary~\ref{cor:NC}. 
Similarly to the proof of Lemma~\ref{lem:complete}, 
the natural transform
\begin{align*}
h_{\alpha~[F]}^{\rm{loc}}\to \mathrm{Def}_{F}^{\n}
\end{align*}
sending triples $(f, \fF, \psi)$ 
to $(\fF, \Lambda/{\bf n}\otimes_{\Lambda}\psi)$
is formally smooth
and isomorphism on $\mathbb{C}[t]/t^2$.  
Hence 
$\widehat{\oO}_{U_i, [F]}^{\n} \cong R_{F}^{\n}$
follows from Lemma~\ref{lem:prohull} and the uniqueness of 
pro-representable hull. 
\end{proof}

\subsection{Partial NC thickening of moduli spaces of sheaves}
Let
$\{U_i^{\n}\}_{i\in \mathbb{I}}$
be a quasi NC structure 
on $M_{\alpha}$
given in Corollary~\ref{cor:NC}, 
and $\eE_i^{\n}$ 
the object in $\Coh(X \times U_{i}^{\n})$
given in (\ref{isom:Ei}). 
For $d\in \mathbb{Z}_{\ge 0}$, 
we set
\begin{align*}
U_i^{d}=(U_i^{\n})^{\le d}, \ 
\eE_i^{d}=(\eE_i^{\n})^{\le d}. 
\end{align*}
The isomorphisms
(\ref{phiU}), (\ref{isom:Ei}) induce isomorphisms
\begin{align*}
\phi_{ij}^{\le d} \colon 
U_j^{d}|_{U_{ij}} \stackrel{\cong}{\to}
U_i^{d}|_{U_{ij}}, \ 
g_{ij}^{\le d} \colon 
\phi_{ij}^{\le d \ast} \eE_i^{d}|_{U_{ij}}
\stackrel{\cong}{\to}
\eE_j^{d}|_{U_{ij}}. 
\end{align*}
Similarly to Subsection~\ref{subsec:partial},
we assume the following: 
\begin{assum}
\label{assum:assum2}
\begin{enumerate}
\item By replacing $\phi_{ij}^{\le d-1}$ if necessary, 
the quasi NC structure 
$\{U_i^{d-1}\}_{i \in \mathbb{I}}$ 
determines an NC structure
$M_{\alpha}^{d-1}$
on $M_{\alpha}$. 
\item By replacing $g_{ij}^{\le d-1}$
if necessary, 
the sheaves $\eE_{i}^{d-1}$ glue to give 
an object $\eE^{d-1} \in \Coh(X \times M_{\alpha}^{d-1})$. 
\end{enumerate}
\end{assum}
We have the following result similar to Corollary~\ref{cor:glue}:
\begin{prop}\label{prop:framed}
Under Assumption~\ref{assum:assum2}, affine 
NC structures $\{U_i^{d}\}_{i \in \mathbb{I}}$
glue to give an NC structure 
$M_{\alpha}^{d}$ on $M_{\alpha}$. 
\end{prop}
\begin{proof}
The result follows from Corollary~\ref{cor:glue}
and Proposition~\ref{prop:isom}. 
\end{proof}
Hence 
we always have a 1-thickening 
$M_{\alpha}^1$ of $M_{\alpha}$. 
By Proposition~\ref{prop:framed},
if the obstruction extending $\eE^{d-1}$ to a $d$-th
order $\eE^d$
vanishes, then
$M_{\alpha}$ admits a
$(d+1)$-th 
global NC
thickening. 
We set
\begin{align*}
\iI^d \cneq 
\Ker \left( 
\oO_{M_{\alpha}}^d \twoheadrightarrow \oO_{M_{\alpha}}^{d-1} \right)
\end{align*}
which is a coherent sheaf on $M_{\alpha}$. 
Similarly to Lemma~\ref{lem:obstr}, 
the obstruction extending $\eE^{d-1}$ to $\eE^d$
lies in $H^2(M_{\alpha}, \iI^d)$. 
In particular if $\dim M_{\alpha} \le 1$, 
then 
$\{U_i^{\n}\}_{i\in \mathbb{I}}$ glue to give 
a global NC structure on $M_{\alpha}$.

\subsection{NC structures on framed moduli spaces of sheaves}\label{subsec:NCframed}
We fix $q\gg p \gg 0$ as in the previous subsections. 
A pair $(F, s)$
for $F \in \Coh(X)$ and 
$s \in \Gamma(F(p))$
is called a \textit{framed sheaf}. 
\begin{defi}
A framed sheaf $(F, s)$ is called 
framed stable if 
$F$ is 
a semistable sheaf
and for any 
proper subsheaf $0 \subsetneq F' \subsetneq F$
which contains $s$, 
the inequality (\ref{def:stab}) is strict. 
\end{defi}
The functor
\begin{align*}
\mM_{\alpha, p}^{\star} \colon 
\sS ch/\mathbb{C} \to \sS et
\end{align*}
is defined by sending a $\mathbb{C}$-scheme 
$T$ to the set of isomorphism classes of 
pairs $(\fF, s)$, where 
$\fF \in \mathfrak{M}_{\alpha}(T)$ and
$s \in \Gamma(\fF(p))$ such that 
the pair $(\fF_t, s_t)$
is framed stable for any $t\in T$. 
An isomorphism 
from $(\fF, s)$ to $(\fF', s')$
is an isomorphism of sheaves
$g \colon \fF \to \fF'$
sending $s$ to $s'$. 
It is well-known that 
$\mM_{\alpha, p}^{\star}$ is represented by a projective 
scheme $M_{\alpha, p}^{\star}$ (cf.~\cite[Section~12]{JS}). 

Let $\star$ be the vertex $\{p\}$ in the quiver $Q_{[p, q]}$. 
For a framed sheaf $(F, s)$, 
the definition of $\Gamma_{[p, q]}$ 
naturally gives the 
framed representation  
$(\Gamma_{[p, q]}(F), s)$
of $(Q_{[p, q]}, I)$. 
\begin{lem}\label{lem:fstable}
There exist 
$q\gg p \gg 0$ such that 
for any semistable sheaf $F$ with Hilbert polynomial $\alpha$, 
a framed sheaf $(F, s)$ is framed stable 
if and only if  
$(\Gamma_{[p, q]}(F), s)$ is framed $\theta$-stable. 
\end{lem}
\begin{proof}
Let $(F, s)$ be a framed sheaf
such that $F$ is semistable with Hilbert polynomial $\alpha$.
By Theorem~\ref{thm:BFHR}, 
the representation
$\Gamma_{[p, q]}(F)$ 
of $(Q_{[p, q]}, I)$ is $\theta$-semistable. 
Moreover 
any subsheaf $0 \neq F' \subsetneq F$ 
with the same reduced Hilbert polynomials 
gives rise to 
the subrepresentation 
$0\neq W \subsetneq \Gamma_{[p, q]}(F)$
with $\theta \cdot \dim W=0$, 
which contains $s$ if $F'$ does. 
Therefore if $(\Gamma_{[p, q]}(F), s)$ is 
framed $\theta$-stable, then $(F, s)$ must be framed stable.

Conversely, suppose that 
$(\Gamma_{[p, q]}(F), s)$ is not framed $\theta$-stable. 
Then there is a subrepresentation $0\neq W\subsetneq \Gamma_{[p, q]}(F)$ 
of $(Q_{[p, q]}, I)$
which contains $s$
such that $\theta \cdot \dim W=0$. 
Let $Q_{[p, q]}^{\dag}$ be the 
quiver with vertex $\{p, q\}$ and 
$\dim_{\mathbb{C}}{\bf m}_{q-p}$-arrows 
from $p$ to $q$. 
By~\cite[Theorem~5.10 (a)]{ACK}, 
the
representation of $Q_{[p, q]}^{\dag}$
\begin{align*}
\Gamma_{[p, q]}^{\dag}(F)=(H^0(F(p)), H^0(F(q)))
\end{align*}
is $\theta^{\dag}=(\theta_p, \theta_q)$-semistable. 
Moreover $W$ 
gives rise to the $Q_{[p, q]}^{\dag}$ subrepresentation
 $W^{\dag}=(W_p, W_q)$
of $\Gamma_{[p, q]}^{\dag}(F)$
with $\theta^{\dag} \cdot \dim W^{\dag}=0$. 
By~\cite[Theorem~5.10 (c)]{ACK}, 
$W^{\dag}$ is given by $\Gamma_{[p, q]}^{\dag}(F')$
for some $0\neq F' \subsetneq F$
having the same reduced Hilbert polynomials. 
As $s \in W_p=H^0(F'(p))$, 
the pair $(F, s)$ is not framed stable. 
\end{proof}
By Lemma~\ref{lem:fstable},
we have the morphism
\begin{align}\label{gstar}
M_{\alpha, p}^{\star} \to 
M_{Q_{[p, q]}, I, \theta}^{\star}(\alpha_{[p, q]})
\end{align} 
sending $(F, s)$ to $(\Gamma_{[p, q]}, s)$. 
Let $M_{[p, q]}^{\star}$
be the open subscheme of 
$M_{Q_{[p, q]}, I, \theta}^{\star}(\alpha_{[p, q]})$
given by the Cartesian square
\begin{align*}
\xymatrix{
M_{[p, q]}^{\star} \ar@{^{(}->}[r] \ar[d] & 
M_{Q_{[p, q]}, I, \theta}^{\star}(\alpha_{[p, q]}) \ar[d] \\
\mathfrak{M}_{[p, q]} \ar@{^{(}->}[r] &
\mathfrak{M}_{Q_{[p, q]}, I, \theta}(\alpha_{[p, q]}). 
}
\end{align*}
Here the bottom morphism 
is the open immersion, 
and 
the right morphism is forgetting the 
framing. 
The morphism (\ref{gstar})
factors through the morphism
\begin{align}\label{gstar2}
\Upsilon^{\star} \colon 
M_{\alpha, p}^{\star} \to M_{[p, q]}^{\star}. 
\end{align}
\begin{lem}\label{lem:fisom}
The morphism $\Upsilon^{\star}$ is an isomorphism. 
In particular, 
the morphism (\ref{gstar}) is an open immersion  
whose image $M_{[p, q]}^{\star}$ 
consists of union of components of 
$M_{Q_{[p, q]}, I, \theta}^{\star}(\alpha_{[p, q]})$. 
\end{lem}
\begin{proof}
We have the commutative diagram
\begin{align*}
\xymatrix{
M_{\alpha, p}^{\star} \ar[r]^{\Upsilon^{\star}} \ar[d] &
M_{[p, q]}^{\star} \ar[d] \\
\mathfrak{M}_{\alpha} \ar[r]^{\Gamma_{[p, q]}} &
\mathfrak{M}_{[p, q]}. 
}
\end{align*}
Here the vertical 
morphisms are forgetting 
the framings. 
By Lemma~\ref{lem:fstable}, the above 
diagram is Cartesian. 
Since the bottom morphism is an isomorphism 
by Theorem~\ref{thm:BFHR}, 
the morphism $\Upsilon^{\star}$ 
is also an isomorphism. 
\end{proof}
Let $(\eE, \iota)$ be the 
universal family of framed stable sheaves
on $M_{\alpha, p}^{\star}$, i.e.
\begin{align*}
\eE \in \Coh(X \times M_{\alpha, p}^{\star}), \ 
\iota \in \Gamma(\eE(p)). 
\end{align*}
We define the functor
\begin{align*}
h^{\star}_{\alpha, p} \colon \nN \to \sS et
\end{align*}
by sending $\Lambda \in \nN$ to the 
set of isomorphism classes of triples 
$(f, (\fF, s), \psi)$: 
\begin{itemize}
\item $f$ is a morphism of schemes
$f \colon \Spec \Lambda^{ab} \to M_{\alpha, p}^{\star}$. 
\item $\fF \in \mathfrak{M}_{\alpha}^{\n}(\Lambda)$
and $s \in \Gamma(\fF(p))$. 
\item $\psi$ is an isomorphism 
$\psi \colon (\fF^{ab}, s^{ab}) \stackrel{\cong}{\to}
f^{\ast}(\eE, \iota)$
of framed sheaves. 
\end{itemize}
We also define 
\begin{align*}
h_{[p, q]}^{\star}
\cneq h_{Q_{[p, q]}, I, \theta}^{\star}(\alpha_{[p, q]})|_{M_{[p, q]}^{\star}} 
\colon 
\nN \to \sS et. 
\end{align*}
Here $h_{Q_{[p, q]}, I, \theta}^{\star}(\alpha_{[p, q]})$
is introduced in (\ref{def:hQI}). 
Since the functor 
$\Gamma_{[p, q]}$ takes 
$\mathfrak{M}_{\alpha}^{\n}$ to $\mathfrak{M}_{[p, q]}^{\n}$, 
we have the natural transform
\begin{align}\label{nat:Gamma}
\Gamma_{[p, q]} \colon h_{\alpha, p}^{\star} \to h_{[p, q]}^{\star}. 
\end{align}
defined in an obvious way. 
\begin{prop}\label{prop:framediso}
The natural transform (\ref{nat:Gamma})
is an isomorphism of functors. 
\end{prop}
\begin{proof}
Similarly to the proof of Lemma~\ref{lem:fisom},
we have the Cartesian diagram
\begin{align*}
\xymatrix{
h_{\alpha, p}^{\star} \ar[r]^{\Gamma_{[p, q]}} \ar[d] &
h_{[p, q]}^{\star} \ar[d] \\
\mathfrak{M}_{\alpha}^{\n} \ar[r]^{\Gamma_{[p, q]}} &
\mathfrak{M}_{[p, q]}^{\n}. 
}
\end{align*}
Here the left vertical arrow is 
sending $(f, (\fF, s), \psi)$ to $\fF$, and the right 
vertical arrow is similar. 
Since the bottom arrow 
is an isomorphism by Proposition~\ref{prop:isom}, 
 the top arrow is also an isomorphism. 
\end{proof}

Finally we obtain the following result: 
\begin{thm}\label{thm:NCframe}
The framed moduli scheme 
$M_{\alpha, p}^{\star}$ has a canonical NC structure 
which represents the functor 
$h_{\alpha, p}^{\star}$. 
\end{thm}
\begin{proof}
The result is an immediate consequence of 
Theorem~\ref{thm:framedQ} and Proposition~\ref{prop:framediso}. 
\end{proof}

\section{Examples}\label{sec:NC4}
In this section, we discuss some
examples of non-commutative thickening 
of moduli spaces of sheaves. 
\subsection{Non-commutative moduli spaces of points}
Let $X$
be a smooth projective variety over $\mathbb{C}$. If we take 
$\alpha \in \mathbb{Q}[t]$ to be the constant function 
$\alpha=1$, then we have 
the isomorphism
\begin{align*}
X \stackrel{\cong}{\to} M_{\alpha}
\end{align*}
sending $x \in X$ to the skyscraper sheaf $\oO_x$. 
In this case, the dg-algebra $\dR \Hom(\oO_x, \oO_x)$ is
quasi isomorphic to the 
exterior algebra
\begin{align*}
\dR \Hom(\oO_x, \oO_x) =
\bigoplus_{i\ge 0} \bigwedge^i T_x X[i]. 
\end{align*}
Therefore we have (cf.~\cite{ESe})
\begin{align*}
R_{\oO_x}^{\n}=
\widehat{T}^{\bullet}((T_x X)^{\vee})/\langle u \otimes v-v\otimes u
 \rangle
\end{align*}
for $u, v \in T_x X$. 
Hence $R_{\oO_x}^{\n}$ is
isomorphic to 
$\widehat{\oO}_{X, x}$, which is a commutative algebra. 
Let $\{U_i^{\n}\}_{i\in \mathbb{I}}$ be a quasi NC structure 
given in Theorem~\ref{main:thm}. 
Then $U_i^{\n}=U_i$ in this case, and 
they are of course
glued to give $X$, i.e. 
the
global
 non-commutative moduli space of points
in $X$ is $X$ itself.

\subsection{Non-commutative moduli spaces of contractible curves}
Let $X$ be a quasi projective 3-fold 
and 
\begin{align*}
f \colon X \to Y
\end{align*}
a
flopping contraction which contracts a single smooth
rational curve $C \subset X$
to a point $p\in Y$. 
Let $\alpha \in \mathbb{Q}[t]$ be the
Hilbert polynomial of 
$\oO_C$, 
and $M_{\alpha}$ the 
commutative moduli 
space of stable sheaves with Hilbert polynomial 
$\oO_C$. 
It is well-known that 
$M_{\alpha}$ is topologically one point, 
consisting of $\oO_C$ (cf.~\cite{Katz}). 
Therefore giving a
quasi NC structure on $M_{\alpha}$ is equivalent to 
giving a NC structure, which is 
equivalent to giving 
an NC complete algebra whose abelization is $\oO_{M_{\alpha}}$. 

It is well-known that 
the normal bundle $N_{C/X}$ is 
given by $\oO_C(a) \oplus \oO_C(b)$
such that
\begin{align*}
(a, b) \in \{(-1, -1), (0, -2), (1, -3)\}. 
\end{align*}
The non-commutative moduli space of 
the object $\oO_C$ was studied by Donovan-Wemyss~\cite{WM}. 
By~\cite{WM}, 
the algebra $R_{\oO_C}^{\n}$ is commutative if and only if 
$C$ is not a $(1, -3)$-curve, 
and in this case $R_{\oO_C}^{\n}$ is isomorphic to 
$\mathbb{C}[t]/t^k$ for some $k\in \mathbb{Z}_{\ge 1}$. 
An example of a $(1, -3)$-curve is given by 
the exceptional locus 
of a crepant small resolution of the 
affine singularity 
$Y=\Spec R_k$, where $R_k$ is defined by 
\begin{align*}
R_k=\mathbb{C}[u, v, x, y]/(u^2+v^2 y=x(x^2 + y^{2k+1})). 
\end{align*}
In this case, 
the algebra $R_{\oO_C}^{\n}$ is given by~\cite[Example~3.14]{WM}
\begin{align*}
R_{\oO_C}^{\n} =\mathbb{C} \langle
x, y \rangle /(xy=-yx, x^2=y^{2k+1}).  
\end{align*}
In this case, the global 
NC structure on $M_{\alpha}$ is given by 
$\Spf R_{\oO_C}^{\n}$.

\subsection{Non-commutative moduli spaces of line bundles}
Let $X$ be a smooth projective variety, and 
$\alpha \in \mathbb{Q}[t]$ the 
Hilbert polynomial of $\oO_X$. 
Then we have 
\begin{align*}
M_{\alpha}=\Pic^0(X)
\end{align*}
where 
$\Pic^0(X)$ is 
the moduli space of line bundles on $X$
with $c_1=0$. 
The moduli space $\Pic^0(X)$ is an abelian 
variety with
dimension $\dim H^1(\oO_X)$. 
Let $\{U_i^{\n}\}_{i \in \mathbb{I}}$
be a quasi NC structure in Corollary~\ref{cor:NC}.
In this case, the moduli space $\Pic^0(X)$
is also interpreted as  
the moduli space of pairs $(\lL, s)$, 
where $\lL \in \Pic^0(X)$ 
and $s$ is an 
isomorphism
\begin{align*}
s \colon \mathbb{C} \stackrel{\cong}{\to} \lL|_{p}.
\end{align*}
Note that the different choices of $s$ 
yield isomorphic pairs $(\lL, s)$. 
The data $s$ behaves like a choice of 
a framing 
in Subsection~\ref{subsec:NCframed}. 
Although we omit a detail here, 
the 
proof similar to 
Proposition~\ref{prop:framed}
shows that $\{U_i^{\n}\}_{i\in \mathbb{I}}$ glue to 
give the NC structure on $\Pic^0(X)$. 

In fact, 
this idea was used by Polischchuk-Tu~\cite{PoTu} 
to give a 
global NC smooth thickening of $\Pic^0(X)$
when $H^2(\oO_X)=0$. 
In general, using 
the notion of algebraic NC connections, 
the global NC structure on $\Pic^0(X)$ was 
also constructed by Polishchuk-Tu~\cite[Section~7.1]{PoTu}, 
satisfying the property of Theorem~\ref{main:thm}.

\subsection{Non-commutative moduli spaces of stable sheaves on K3 surfaces}
Let $X$ be a smooth projective K3 surface over $\mathbb{C}$, 
and suppose that $\alpha \in \mathbb{Q}[t]$ is primitive. 
By the result of Mukai~\cite{Mu2}, 
any connected component of 
the moduli space $M_{\alpha}$ is
a holomorphic symplectic manifold. 
For a stable sheaf $[F] \in M_{\alpha}$, 
the dg-algebra $\dR \Hom(F, F)$ is 
known to be formal (cf.~\cite[Proposition~1.3]{ZiZha}), i.e. 
there is a quasi-isomorphism of dg-algebras
\begin{align*}
\dR \Hom(F, F) \cong (\Ext^{\ast}(F, F), m_2). 
\end{align*}
In particular the higher $A_{\infty}$-products
$m_n$ 
of the minimal model (\ref{minimal})
vanish for $n\ge 3$. 
Also the multiplication
\begin{align*}
m_2 \colon \Ext^1(F, F) \times \Ext^1(F, F) \to
\Ext^2(F, F) \cong \mathbb{C}
\end{align*}
gives a holomorphic symplectic form on $M_{\alpha}$. 
Therefore
by choosing
a suitable basis of $\Ext^1(F, F)$, 
the algebra $R_F^{\n}$ is given by
\begin{align}\label{K3:FR}
R_F^{\n}=\frac{\mathbb{C}\langle \langle x_1, x_2, \cdots, x_{2m-1}, x_{2m} 
\rangle \rangle}{\langle [x_1, x_2]+ \cdots + [x_{2m-1}, x_{2m}] \rangle}. 
\end{align}
Let $\{U_i^{\n}\}_{i\in \mathbb{I}}$
be a quasi NC structure on $M_{\alpha}$
given in Corollary~\ref{cor:NC}. 
We do not know whether $\{U_i\}_{i\in \mathbb{I}}$
glue to give a global NC structure on $M_{\alpha}$. 
However by Proposition~\ref{prop:framed}, we at least know 
that the 1-thickenings
$\{U_i^{1}\}_{i\in \mathbb{I}}$ glue to give a 
NC structure
$M_{\alpha}^1$ on $M_{\alpha}$. 
By (\ref{K3:FR}), one 
of the gluing is given by the sheaf of algebras
\begin{align*}
\oO_{M_{\alpha}^1}=
\oO_{M_{\alpha}} \oplus \left(\Omega_{M_{\alpha}}^2/\oO_{M_{\alpha}}\right). 
\end{align*}
Here $\oO_{M_{\alpha}^1} \subset \Omega_{M_{\alpha}}^2$ is 
given by the 
holomorphic symplectic form on $M_{\alpha}$, 
and the algebra structure on $\oO_{M_{\alpha}}^1$
is 
given by 
\begin{align*}
(x, f) \cdot (y, g)=(xy, xg+fy+dx \wedge dy). 
\end{align*}

\subsection{Non-commutative thickening of Hilbert schemes of points}\label{subsec:hilb}
Let $X$ be a projective scheme over $\mathbb{C}$
and take 
$\alpha \in \mathbb{Q}[t]$ 
to be the constant 
function $\alpha=n$ for $n\in \mathbb{Z}_{\ge 1}$. 
Let $M_{\alpha, p}^{\star}$ be the moduli 
space of framed stable 
sheaves given in Subsection~\ref{subsec:NCframed}.
Then $M_{\alpha, p}^{\star}$ is independent of $p$, and 
we have the isomorphism
\begin{align*}
\Hilb^n(X) \stackrel{\cong}{\to} M_{\alpha, p}^{\star}
\end{align*} 
sending $Z \subset X$ to 
$(\oO_Z, s)$ where 
$s \in H^0(\oO_Z)$ is the canonical surjection
$\oO_X \twoheadrightarrow \oO_Z$. 
Here $\Hilb^n(X)$ is the Hilbert scheme of $n$-points, 
parameterizing zero
dimensional subschemes $Z \subset X$ with length $n$. 
By Theorem~\ref{thm:NCframe}, 
the Hilbert scheme of points $\Hilb^n(X)$ has a canonical 
NC structure. 
For example, 
one can check that the
NC structure on $\Hilb^2(\mathbb{C}^2)$
induced by the open immersion 
$\Hilb^2(\mathbb{C}^2) \subset \Hilb^2(\mathbb{P}^2)$
coincides with the one 
given in Subsection~\ref{subsec:anex}. 
\begin{rmk}
If $X$ is non-singular with $\dim X \ge 2$ and $H^1(\oO_X)=0$, 
then $\Hilb^n(X)$ is also
regarded as a moduli space of 
unframed sheaves
by associating
a zero dimensional subscheme $Z \subset X$
with the ideal sheaf $I_Z \subset \oO_X$. 
However a quasi NC structure on $\Hilb^n(X)$
given as the unframed moduli of sheaves 
is in general different from the
above NC structure. 
\end{rmk}

\providecommand{\bysame}{\leavevmode\hbox to3em{\hrulefill}\thinspace}
\providecommand{\MR}{\relax\ifhmode\unskip\space\fi MR }
\providecommand{\MRhref}[2]{%
  \href{http://www.ams.org/mathscinet-getitem?mr=#1}{#2}
}
\providecommand{\href}[2]{#2}

%\bibliographystyle{amsalpha}
%\bibliography{math}

Kavli Institute for the Physics and 
Mathematics of the Universe, University of Tokyo,
5-1-5 Kashiwanoha, Kashiwa, 277-8583, Japan.

\textit{E-mail address}: yukinobu.toda@ipmu.jp

\end{document}